\pdfoutput=1
\documentclass[11pt]{article}
\usepackage[utf8]{inputenc}
\usepackage{geometry}

\usepackage[compact]{titlesec}
\titlespacing{\section}{0pt}{2ex}{1ex}
    \titlespacing{\subsection}{0pt}{1ex}{0ex}
    \titlespacing{\subsubsection}{0pt}{0.5ex}{0ex}

\usepackage{graphicx}
\graphicspath{ {./images/} }
\usepackage{amsmath,xpatch} % Equations
\usepackage{amssymb} % Equations
\usepackage{listings}
\usepackage{verbatim}
\usepackage[dvipsnames]{xcolor}
\usepackage{color}
\usepackage[vcentermath]{youngtab}
\usepackage{amsthm}
\usepackage{tikz}
\usepackage{lscape}
\usepackage{hyperref}
\usepackage{authblk}
\usepackage{dsfont,  mathabx}
\usepackage{float}
\usepackage[linesnumbered,lined,commentsnumbered,ruled]{algorithm2e}

\makeatletter
\xpatchcmd{\proof}{\topsep0\p@\@plus0\p@\relax}{}{}{}
\makeatother

\usepackage{my_preamble}

\usepackage[compact]{titlesec}
\titlespacing{\section}{0pt}{1ex}{1ex}
    
\hypersetup{
    colorlinks=true,
    linkcolor=magenta,
    filecolor=magenta,      
    urlcolor=magenta,
    citecolor=magenta
}
\title{\bf Maximum information divergence from \\linear and toric models}
\author[1]{Yulia Alexandr}
\author[2]{Serkan Ho\c{s}ten}
\affil[1]{University of California, Berkeley}
\affil[2]{San Francisco State University}

\date{}

\begin{document}
\maketitle
\begin{abstract}
We study the problem of maximizing information divergence from a new perspective using logarithmic Voronoi polytopes. We show that for linear models, the maximum is always achieved at the boundary of the probability simplex.  For toric models, we present an algorithm that combines the combinatorics of the chamber complex with numerical algebraic geometry. We pay special attention to reducible models and models of maximum likelihood degree one.

% We study the question of maximizing information divergence from discrete statistical models. The maximizers are always vertices of logarithmic Voronoi polytopes at points on a model. We show that for linear models, the maximum is always achieved at the boundary of the simplex. For toric models, we present a new algorithm for computing the maximum divergence and the maximizers, which combines the combinatorics of the chamber complex with numerical algebraic geometry.   We then focus on toric models with ML degree one by completely characterizing the general box model and providing an upper bound for the divergence of the trapezoid model.
    
\end{abstract}
\section{Introduction} \label{introduction}
Let $\M \subset \Delta_{n-1}$ be a statistical model 
where 
$$ \Delta_{n-1} \, = \, \left\{ p = (p_1, \ldots, p_n) \, : \, \sum_{i=1}^n p_i =1 \mbox{ and } p_i \geq 0, \, i=1, \ldots n \right\} $$
is the probability simplex of dimension $n-1$. Given two points $p,q\in\Delta_{n-1}$ with $\supp(p) \subseteq \supp(q)$, the information divergence or Kullback-Leibler divergence of $p$ and $q$ is defined as $$D(p\parallel q) \, := \, \sum_{i=1}^n p_i\log \left(\frac{p_i}{q_i} \right).$$ 
We use the convention that $0 \log 0 = 0 \log(0/0) = 0$ and $D(p \parallel q) = +\infty$ if $\supp(p) \not \subseteq \supp(q)$. For fixed $q$, the function $D(\cdot \parallel q)$
is strictly convex. The information divergence (or just divergence) from $p \in \Delta_{n-1}$ to $\M$
is 
$$ D_\M(p) \, := \, \min_{q \in \M} \, D(p \parallel q).$$ In this paper, we study $D(\M) = \max_{p \in \Delta_{n-1}} \, D_\M(p)$ and the points which achieve
$D(\M)$ when $\M$ is a linear or a discrete exponential (toric) model. 

\subsection{Related prior work}
The problem of determining $D(\E)$ and studying the maximizers of the divergence function from an exponential family $\E \subset \Delta_{n-1}$ was first posed by Ay \cite{Ay02} who computed the gradient of $D_\E(p)$. The exponential family $\M$  of probability distributions of independent random variables $X_i$, $i=1, \ldots, m$ with state spaces
$[d_i] := \{1, \ldots, d_i\}$ is known as an independence model. In this case, $D_\M(p)$ is the multi-information, and $D(\M) \leq \sum_{i=1}^{m-1} \log(d_i)$ where $2 \leq d_1 \leq d_2 \leq \cdots \leq d_m$~\cite{AK06}. In the same work, the structure of the global maximizers of the multi-information when the above bound is achieved was also determined. Subsequently, Mat\'{u}\v{s} has computed the optimality conditions
for $D_\E(p)$ for any exponential model $\E \subset \Delta_{n-1}$ \cite{Mat07}. We will use these conditions heavily.
Rauh's dissertation \cite{RauhThesis} as well
as his work in \cite{Rauh11} gave algorithms
to compute $D(\M)$ for a discrete exponential family $\M$. These algorithms have two components: a combinatorial step followed by an algebraic step, both of which can be challenging. Nevertheless, they were capable of computing the maximum multi-information to an independence model with $d_1=2$ and
$d_2=d_3=3$, the smallest case where the aforementioned bound is not attained. We will provide another algorithm in
the same spirit with combinatorial and algebraic steps. Finally, the literature contains results on
the maximum divergence from certain hierarchical models \cite{Mat09}, partition models \cite{Rauh13}, naive Bayes models and restricted Boltzmann machines \cite{MRA13}.

\subsection{Preliminaries and summary of results}
Let $\mathfrak{X}$ be a finite set of cardinality $n$
and let $A$ be a $d \times n$ matrix with entries in $\RR$. With respect to the reference measure $\omega(x), \, x \in \mathfrak{X}$, the exponential family $\E = \E_{\omega,A}$ consists of the positive probability
distributions in $\Delta_{n-1}$ of the form
$$ P_{\theta}(x) = \frac{\omega(x)}{Z_{\theta}} \exp{ \left( \sum_{i=1}^d \theta_i A_{i,x} \right)},$$
where $A_i$ is the $i$th row of $A$ and $Z_\theta$ is 
the normalizing constant. Here $\theta_i \in \RR$ and $\overline{\E}$, the Euclidean closure of $\E$ in $\Delta_{n-1}$, will be referred to as the extended exponential family.  Usually we will identify 
$\mathfrak{X}$ with $[n]$ and write $p_i$ and $\omega_i$ instead of $P(x)$ and $\omega(x)$, respectively. 

In this paper, we consider {\it discrete} exponential families because of the bridge to toric geometry and algebraic statistics \cite{DrtonSturmfelsSullivant2009LecturesOnAlgebraicStatistics, AlgStatBook}. This means that $A$ is a matrix with {\it integer} entries. Since without loss of generality we can 
assume that the row span of $A$ 
$$ A = \left( \begin{array}{cccc} 
        a_1 & a_2 & \cdots & a_n \end{array} \right)$$
contains $(1,1,\ldots, 1)$,         
we will take  the columns $a_j \in \NN^d$, $j=1, \ldots, n$ and fix the first row of $A$ to be the row of all ones.
The toric variety $X_{\omega, A}$ is the Zariski closure in $\CC^{n}$ of the image of the algebraic torus
$(\CC^*)^d$ under the monomial map
$\Psi \, : \, (\CC^*)^d \longrightarrow \CC^n$ given by 
$$ z = (z_1, \ldots, z_d) \mapsto (\omega_1 z^{a_1}, \, \omega_2 z^{a_2}, \, \ldots, \, \omega_n z^{a_n}).$$
Because of the assumption on the first row of $A$, we can also view $X_{\omega,A}$ as a toric variety in the projective space $\PP^{n-1}$.  The following theorem connects exponential families and toric varieties.

\begin{theorem}\label{thm:discrete=toric} \cite[Theorem 3.2]{GMS06} 
The extended exponential family $\overline{\E}_{\omega, A}$ is equal to $X_{\omega, A} \cap \Delta_{n-1}$.
\end{theorem}
Therefore, we will refer to discrete
exponential families as toric models. We will denote them by $\M_{\omega,A}$ or just $\M_A$. 

Given a toric model $\M_A$ and a fixed $u \in \Delta_{n-1}$,  the minimum $D_{\M_A}(u)$ is attained at a unique
point $q \in \M_A$. It is known as the \textit{maximum likelihood estimate} (MLE) of $u$. Birch's Theorem (see \cite[Theorem 4.8]{Lau96}, \cite[Proposition 2.1.5]{DrtonSturmfelsSullivant2009LecturesOnAlgebraicStatistics}, \cite[Theorem 1.10]{ASCB05}) states that 
the maximum likelihood estimate of $u$ is equal 
to the unique point in the intersection $$\M_A \cap \{p\in\Delta_{n-1}: Au=Ap\}.$$
The second term in this intersection is the polytope 
$Q_u := \{p\in\Delta_{n-1}: Au=Ap\}$. If $q \in \M_A$ is the MLE of $u$, by Birch's Theorem $Q_q = Q_u$. We will call this polytope the \textit{logarithmic Voronoi polytope} at $q$ following \cite{AH20logarithmic}.  For an arbitrary model $\M \subset \Delta_{n-1}$ and $q \in \M$, the logarithmic 
Voronoi cell at $q$ consists of all points $u \in 
\Delta_{n-1}$ such that a maximum likelihood estimate of $u$ is $q$. Logarithmic Voronoi cells are always convex sets \cite[Proposition 4]{AH20logarithmic}, and when the model is linear or toric, they are polytopes \cite[Theorem 9-10]{AH20logarithmic}. 
\begin{prop} \label{prop:vertex-does-it}
Let $\M \subset \Delta_{n-1}$ be a linear or a toric model and let $q \in \M$. Then the maximum of $D_\M(u)$ restricted to the logarithmic Voronoi polytope $Q_q$ is achieved at a vertex of $Q_q$. The maximizers are a subset of the vertices in $Q_q$.
\end{prop}
\begin{proof}
As we observed above $D(u\parallel q)$ is strictly convex in $u$ over $\Delta_{n-1}$; see for instance \cite[Proposition 2.14 (iii)]{RauhThesis}. The result follows since $Q_q$ is a convex polytope. 
\end{proof}
\begin{cor}\cite[Proposition 3.2]{Ay02} \label{cor:support-of-maximizer}
Let $\M_A \subseteq \Delta_{n-1}$ be a toric model
where $A \in \NN^{d \times n}$ and $\mathrm{rank}(A)=~d$. If
$p$ is a maximizer of the information divergence
then $|\supp(p)| \leq d = \dim(\M_A) + 1$. 
\end{cor}
\begin{proof}
If $q \in \M_A$ is the MLE of $p$, then $p$ is a vertex of $Q_q = \{u \in \Delta_{n-1} \, : \, Au = Aq\}$. Any vertex of $Q_q$ is a basic feasible solution to the system $Au = Aq$. In other words, it 
is of the form $p=(p_B, p_N)$ where $p_N = 0$ and 
$Bp_B = Aq$ with $B$ a $d \times d$ invertible submatrix of $A$. This shows $|\supp(p)| \leq d$.
\end{proof}

The paper is organized as follows. In Section \ref{sec:linear}, we focus on the maximum divergence to linear models. Theorem \ref{thm:linear-vertex} proves that the maximum divergence to a linear model will always be achieved at a vertex of the logarithmic Voronoi polytope at a vertex of the model itself. In Section \ref{sec:crit-pts}, we focus on identifying the critical points of information divergence to toric models. Theorem \ref{thm:critical-projection} gives a necessary and sufficient condition for a vertex of a logarithmic Voronoi polytope to be a critical point. In Section \ref{sec:chamber-complex}, we define the chamber complex of a toric model and describe how it determines the combinatorial type of logarithmic Voronoi polytopes in Theorem \ref{thm:vertices}. We then present an algorithm for maximizing information divergence from a toric model, which utilizes the combinatorics of the chamber complex and numerical algebraic geometry. We make our code for several parts of the algorithm available on Github. \footnote{\url{https://github.com/yuliaalexandr/maximizing-divergence}} Section \ref{sec:decomposable} is devoted to reducible toric models and a decomposition theory of their logarithmic Voronoi polytopes. Theorem \ref{thm:decomposition-theorem} provides a way to reconstruct logarithmic Voronoi polytopes of a reducible model $\M$ from the logarithmic Voronoi polytopes of the models induced by the the reduction of $\M$. Section \ref{sec:comparing-divergences} then explains how to use this decomposition to obtain and bound information divergence to reducible models. Finally, Section \ref{sec:ml-deg-one} studies divergence from toric models of ML degree one. After revisiting the multinomial model in Theorem \ref{thm:twisted-veronese-max-divergence}, we generalize the results to the box model by establishing the maximum divergence and characterizing the set of maximizers in Theorem \ref{thm:box-max-divergence}. Theorem \ref{thm:trapezoid-upper-bound} establishes an upper bound for divergence to the trapezoid model.

\section{Maximum divergence from linear models}\label{sec:linear}
Let $\M$ be a $d$-dimensional linear model in $\Delta_{n-1}$ given by an $n\times d$ matrix $B$ with rows $b_1,\ldots,b_n$ which sum to the zero vector and a vector $c\in\Delta_{n-1}$. That is, $\M$ is the image of the linear map
$$f: \Theta\to \Delta_{n-1}: (x_1,\ldots, x_d)\mapsto (c_1-\langle b_1, x\rangle,\ldots, c_n-\langle b_n, x\rangle).$$
We wish to find $D(\M)$ and the points $ p \in \Delta_{n-1}$ at which the information divergence
$D_\M(p)$ from the linear model is maximized.  By Proposition \ref{prop:vertex-does-it}, 
$$ D(\M) = \max_{q \in \M} \, \max_{p \in Q_q} D(p \parallel q). $$ 
 Hence, the maximum is achieved  at some of the vertices of the logarithmic Voronoi cell $Q_q$ at $q$. The vertices of $Q_q$ at $q=f(x)$ are given by the co-circuits of $B$ and can be expressed as functions in $q$ (or the parameters $x$) \cite[Proposition 2]{Alexandr}. Here, by a co-circuit of $B$ we mean a nonzero $z \in \RR^n$ of minimal support so that $z^TB = 0$. Each co-circuit $z$ of $B$ such that $\langle z,q\rangle=\sum_{i=1}^n z_iq_i=1$ defines a vertex
$V_z(q)=(z_1q_1,\ldots,z_nq_n)$ of $Q_q$. Note that the choice of the co-circuit representative does not depend on the point $q$, i.e. we may always choose the representative $z$ such that $\langle z,q \rangle =1$ for all $q\in \M$ simultaneously. Indeed, let $y$ be some co-circuit of $B$. We wish to find $k\in\RR$ such that $z=ky$ has the property $\langle z, q\rangle=1$ for all $q\in \M$. Since $q=c-Bx$ for some $x\in\Theta$, we have that $$1=\langle z, q\rangle= k\langle y, c-Bx\rangle =k\langle y, c\rangle.$$
 Hence, $z=ky$ where $k=1/\langle y,c\rangle$ is the desired co-circuit representative. For every such co-circuit we wish to maximize the information divergence over all $q\in \M$. We then compare the maximum divergences over all such co-circuits to find the global maximum.
\begin{lemma}\label{lem:linear-divergence}
Let $\M$ be a linear model defined by the matrix $B$ and the vector $c$. 
     For a fixed co-circuit $z$ of $B$, the information divergence $D(V_z(q) \parallel  q)$ is linear in $q \in \M$.
\end{lemma}
\begin{proof}
$$D(V_z(q) \parallel q)=\sum_{i = 1}^n (z_iq_i)\log (z_iq_i/q_i)=\sum_{i=1}^n (z_i\log(z_i))q_i.$$
\end{proof}
Hence, for each co-circuit $z$, we are maximizing a linear function over the polytope $\M$. 
We summarize this in the following result. 
\begin{theorem} \label{thm:linear-vertex}
    The maximum divergence of a linear model $\M$ is always achieved at a vertex of the logarithmic Voronoi polytope $Q_q$ where $q$ itself is a vertex of $\M$.    
\end{theorem}
\begin{remark}
A particular kind of discrete exponential family that is also a linear model is a partition model. The information divergence from partition models have been studied in \cite{MA03}. A result similar to Theorem \ref{thm:linear-vertex} is Proposition 2 in this reference.
\end{remark}
Theorem \ref{thm:linear-vertex} can be used to obtain compact formulas for maximum divergence for special families of linear models, such as the one below.
\begin{cor}
Let $\M$ be a one-dimensional linear model in $\Delta_{3}$ given by $B=[-a,-b,b,a]^T$, $a,b>0$ and $c=(\frac{1}{4},\frac{1}{4},\frac{1}{4},\frac{1}{4})$. Then $D(\M)=\log\left(\frac{4\max\{a,b\}}{a+b}\right)$, maximized at two vertices of~$\Delta_{3}$.
\end{cor}
\begin{proof}
    Without loss of generality, assume that $a>b$. Then the model is parametrized as $f: x\mapsto (ax+1/4, bx+1/4, -bx+1/4, -ax+1/4)$. The two vertices of the model are $v_1=f(-\frac{1}{4a})$ and $v_2=f(\frac{1}{4a})$. Each logarithmic Voronoi polytope is a quadrangle, so the matrix $B$ has four co-circuits which parameterize the four vertices of this polytope at a general point $q=f(x):$
    \begin{align*}
&V_1(x)=(0, 2bx + 1/2, -2bx + 1/2, 0)\\
&V_2(x)=(0, (4abx + a)/(a + b), 0, (b- 4abx)/(a + b))\\
&V_3(x)=(2ax + 1/2, 0, 0, -2ax + 1/2)\\
&V_4(x)=((4abx + b)/(a + b), 0, (a-4abx)/(a + b), 0).
    \end{align*}
Note that $D(V_1(x)||f(x))=D(V_3(x)||f(x))=\log(2)$ for all $x\in\left[-\frac{1}{4a},\frac{1}{4a}\right]$. On the other hand, \begin{align*}
    &D\left(V_2\left(-\frac{1}{4a}\right) \parallel v_1 \right)=\frac{{\left(a - b\right)} \log\left(\frac{4 \, a}{a + b}\right) + 2 \, b \log\left(\frac{4 \, b}{a + b}\right)}{a + b}<\log\left(\frac{4 \, a}{a + b}\right)= D\left(V_2\left(\frac{1}{4a}\right) \parallel v_2 \right)\\
    &D\left(V_4\left(-\frac{1}{4a}\right) \parallel v_1 \right) =\log\left(\frac{4 \, a}{a + b}\right)>\frac{{\left(a - b\right)} \log\left(\frac{4 \, a}{a + b}\right) + 2 \, b \log\left(\frac{4 \, b}{a + b}\right)}{a + b}=D\left( V_4\left(\frac{1}{4a}\right) \parallel v_2 \right).
\end{align*}
Hence, the maximum divergence $\log\left(\frac{4 \, a}{a + b}\right)$ is achieved at the two vertices $V_2\left(\frac{1}{4a}\right)=(0,1,0,0)$ and $V_4\left(-\frac{1}{4a}\right)=(0,0,1,0)$. The proof for $b>a$ is identical.
\end{proof}

\begin{example}
    Consider the $1$-dimensional linear model $\M$ inside $\Delta_3$ given by 
    $B=[-2,-1,1,2]^T$ and $c =(1/4,1/4,1/4,1/4).$ It is a line segment in $\Delta_3$ with the vertices $v_1=f(-1/8)=(0, 1/8, 3/8, 1/2)$ and $v_2=f(1/8)=(1/2, 3/8, 1/8, 0)$. The global maximum divergence $\log(8/3)$ is achieved at $V_4(-1/8)=(0, 0, 1, 0)$ and $V_2(1/8)=(0, 1, 0, 0)$. This is illustrated in Figure \ref{fig:boundary-linear}.
\begin{figure}[H]
    \centering
    \includegraphics[width=0.4\textwidth]{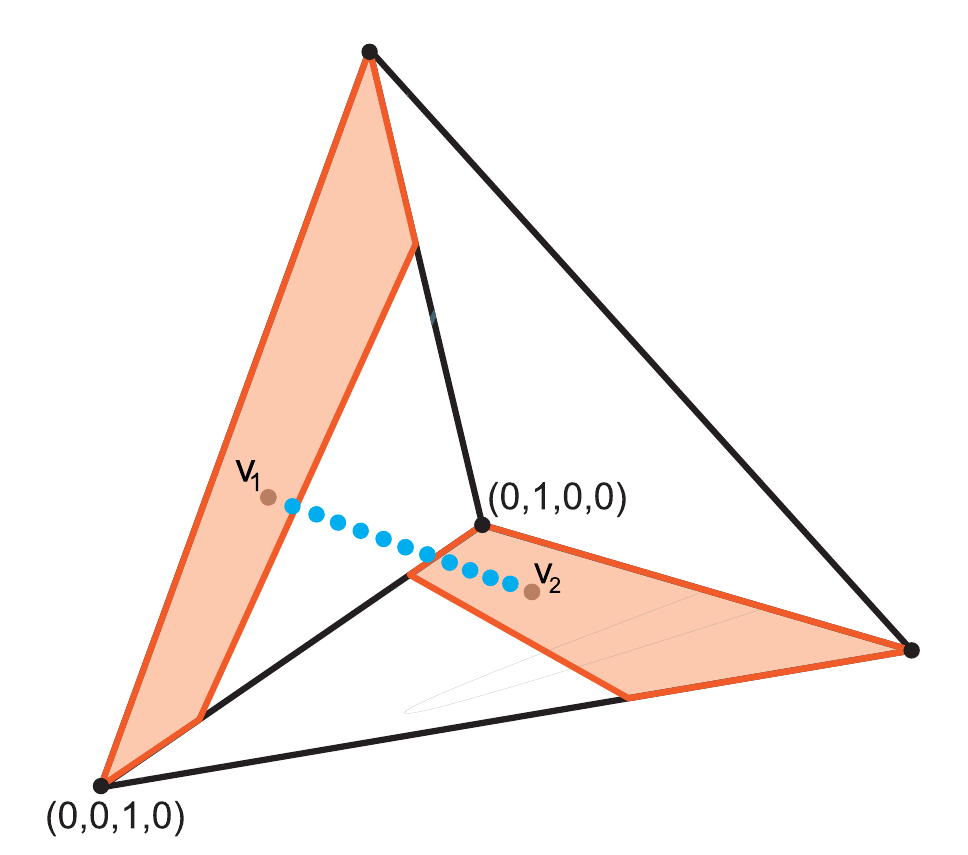}
    \caption{Linear model given by $B=[-2,-1,1,2]^T$.}
    \label{fig:boundary-linear}
\end{figure}
\end{example}

As we close this section we wish to emphasize two relevant facts about logarithmic Voronoi polytopes of 
linear models. First, for all points $q$ that are in the relative interior of $\M$ the combinatorial type 
of $Q_q$ is the same \cite[Corollary 4]{Alexandr}. Moreover, if $\M$ is the transversal intersection
of an affine subspace with $\Delta_{n-1}$, {\it all} $Q_q$, including the ones at boundary points of $\M$, have the same combinatorial type \cite[Theorem 9]{Alexandr}. The results in this section help us identify those logarithmic Voronoi polytopes of just 
one combinatorial type (at least for generic $\M$) potentially containing a vertex attaining $D(\M)$. 
This phenomenon carries over to the toric case where we need to account for the fact that the logarithmic Voronoi polytopes have more than one (but finitely many) combinatorial types. In the next section, we will review results that will be useful in locating vertices 
of logarithmic Voronoi polytopes of the same combinatorial type that potentially maximize the information divergence. Then in Section \ref{sec:chamber-complex} we will see how to parameterize the different combinatorial types and how this helps develop an algorithm to compute $D(\M)$.

\section{Critical points of information divergence to toric models}\label{sec:crit-pts}
 
In the rest of the paper, we will work only with toric models $\M_A$ introduced in Section \ref{introduction}. For a face $F$ of a given polytope $Q$, we define the support of $F$ as the union of the supports of the vertices on $F$ and denote it by $\supp(F)$.  We start with a definition that will pave the path for characterizing the critical points of the function
$D_\M(\cdot)$. 
\begin{definition} \label{def:complementary}
Let $Q_q$ be a logarithmic Voronoi polytope at a point on a toric model $\M_A \subset \Delta_{n-1}$. A vertex $v$ of $Q_q$ is \textit{complementary} if there exists a face $F$ of $Q_q$ such that $\supp(F) = [n] \setminus \supp(v)$. We call $F$ the complementary face of $v$.
\end{definition} 

\begin{definition} \label{def:projection-point}
 Let $\M_A \subset \Delta_{n-1}$ be a toric model and 
 let $p$ be a point in $\Delta_{n-1}$ whose MLE is
 $q$ with $\supp(q)=[n]$. We say that $p$ is a projection point if 
$$p_i=\begin{cases}\frac{q_i}{\sum_{j\in\supp(p)}q_j}&\text{ if }i\in\supp(p)\\ \;\;\;\;\;\;\;\,0&\text{ otherwise}.\end{cases}$$
\end{definition}
\begin{remark}
We can relax the condition for the full support of 
the MLE in the above definition. In this case,
we need to consider MLEs that are in the extended
exponential family, namely, those that are on $\M_A$
and on a proper face $\Gamma$ of $\Delta_{n-1}$. However, these can be separately treated by focusing 
on the toric model $\M_{A_\Gamma} \subset \Gamma$ where $A_\Gamma$ consists of the columns $a_i$ of $A$ with $i \in \supp(\Gamma)$. 
\end{remark}
\begin{theorem}  \label{thm:critical-projection}
If $p$ is a local maximizer of $D_{\M_A}$ then $p$ is a projection point.  Moreover, every such projection point is a complementary vertex of $Q_q$ where $q$ is the MLE of $p$. A complementary vertex $v$ of $Q_q$
with the complementary face $F$
is a projection point if and only if the line passing
through $v$ and $q$ intersects the relative interior
of $F$.
\end{theorem}
\begin{proof}
The first statement is proved in   \cite[Theorem 5.1]{Mat07}.
Since $p$ is a local maximizer it needs to be a vertex of $Q_q$. The point $\tilde{p}$ defined by
$$\tilde{p}_i=\begin{cases}\frac{q_i}{\sum_{j \not \in\supp(p)}q_j}&\text{ if }i \not \in\supp(p)\\ \;\;\;\;\;\;\;\,0&\text{ otherwise}\end{cases}$$
is obtained by $\tilde{p} = p + \frac{1}{\sum_{j \not \in \supp(p)} q_j}[q - p]$ where $[q-p]$ is a vector parallel
to the line through $p$ and~$q$. The support of $\tilde{p}$
is precisely $[n] \setminus \supp(p)$ and therefore it
is contained in the interior of a face $F$ with 
identical support. Hence $p$ is a complementary vertex and the last statement follows.
\end{proof}
\begin{example}
The binomial model of size $3$ is the set of probability 
distributions on $\mathfrak{X} = \{0,1,2,3\}$ parametrized as
$$q_j =  \binom{3}{j} \theta^j(1-\theta)^j, \,\, j=0,1,2,3.$$
This is a one-dimensional toric model that describes the experiment of flipping a coin with the bias $\theta$ three times. The matrix $A$ can be taken to be
$$ A \, = \, \left(\begin{array}{cccc} 1 & 1 &  1 & 1 \\
                                       0 & 1 &  2 & 3 \end{array} \right).$$
For $p = (p_0, p_1, p_2, p_3) \in \Delta_3$ the MLE is given
by
\begin{align*}
  q_0 &= \frac{1}{27}(3p_0+2p_1+p_2)^3, \\
  q_1 &= \frac{1}{9}(p_1+2p_2+3p_3)(3p_0+2p_1+p_2)^2, \\
  q_2 &= \frac{1}{9}(p_1+2p_2+3p_3)^2(3p_0+2p_1+p_2), \\
  q_3 &= \frac{1}{27} (p_1+2p_2+3p_3)^3.
\end{align*}
The logarithmic Voronoi polytopes are of the form
$Q_b = \{u \in \Delta_3 \, : \, u_1 + 2u_2 + 3u_3 = b\}$
where $0< b < 3$. For $0<b<1$ and $2<b<3$ these polytopes are  triangles. The first kind has vertices  with supports 
$\{0,1\}$, $\{0,2\}$, and $\{0,3\}$. The vertices of 
the second kind have supports $\{0,3\}$, $\{1,3\}$, and $\{2,3\}$. None of these triangles have a complementary vertex. 
When $b=1$ and $b=2$, $Q_b$ is still a
triangle: the supports of the vertices of $Q_1$ are
$\{1\}$, $\{0,2\}$, and $\{0,3\}$. Those of $Q_2$ are 
$\{0,3\}$, $\{1,3\}$, and $\{2\}$. 
In $Q_1$, the vertex $(0,1,0,0)$ is a projection
point with divergence $\log \frac{9}{4}$. In $Q_2$,
the vertex $(0,0,1,0)$ is a projection point with
the same divergence. 
The logarithmic Voronoi polytopes for $1<b<2$ are quadrangles with vertex supports $\{0,2\}, \{0,3\}, \{1,2\}$, and $\{1,3\}$. Therefore each vertex is a complementary vertex where the corresponding complementary face $F$ is a vertex itself. Among all these, we find projection points only when $b=3/2$. The vertices
of $Q_{\frac{3}{2}}$ are 
$(\frac{1}{4}, 0, \frac{3}{4}, 0)$, $(\frac{1}{2}, 0, 0, \frac{1}{2})$, $(0, \frac{1}{2}, \frac{1}{2}, 0)$, and $(0, \frac{3}{4}, 0, \frac{1}{4})$. All are projection points
with the MLE $q=(\frac{1}{8}, \frac{3}{8}, \frac{3}{8}, \frac{1}{8})$ which is the intersection of the diagonals of the quadrangle. The divergences from each vertex to this  binomial model are $\log(2), 2\log(2), 2\log(2)-\log(3)$, and $\log(2)$, respectively. Therefore $(\frac{1}{2}, 0, 0, \frac{1}{2})$ is the unique global maximizer attaining $D(\M) = 2 \log(2)$.
\end{example}

\begin{cor} \cite[Section VI]{Rauh11} \label{cor:codim-1}
 Let $\M_A$ be a codimension one toric model in
 $\Delta_{n-1}$, i.e., let $\mathrm{rank}(A)= d = n-1$. 
 Then there are exactly two projection points and at most two global maximizers of $D_{\M_A}$. 
\end{cor}
\begin{proof}
The toric variety $X_A$ is defined by a single equation
which we can assume is of the form $x_1^{u_1}x_2^{u_2} \cdots x_r^{u_r} - x_{r+1}^{u_{r+1}} \cdots x_n^{u_n}$
where $\sum_{i=1}^r u_i = \sum_{j=r+1}^n u_j$. 
The one-dimensional $\ker(A)$ is spanned by 
$(u_1, \ldots, u_r, -u_{r+1}, \ldots, -u_n)$, and all logarithmic Voronoi polytopes are one-dimensional whose
affine span is parallel to $\ker(A)$. Since each such polytope has exactly two vertices, the line through these vertices
always intersects $\M_A$. Hence, for these vertices to be projection points, we only need to make sure that they have complementary support. This can 
only happen if the vertices are $p=(p_1, \ldots, p_r, 0, \ldots, 0)$ and $\tilde{p} = (0, \ldots, 0, \tilde{p}_{r+1},\ldots, \tilde{p}_n)$ where
$p_i = \frac{u_i}{\sum_{i=1}^r u_i}$ for $i=1,\ldots, r$
and $\tilde{p}_j = \frac{u_j}{\sum_{j=r+1}^n u_j}$ for $j=r+1, \ldots, n$. Both points are projection points and
either one or both of them are global maximizers of $D_{\M_A}$.
\end{proof}

\begin{example}\label{ex:two-binary-RV}
Let $X$ and $Y$ be two independent binary random variables. 
The set of joint probability distributions $q_{ij} = 
\Prob(X=i, Y=j)$ with $i,j \in \{0,1\}^2$ is parametrized
by $q_{ij} = a_ib_j$. This toric model $\M_A \subset \Delta_3$ has codimension one and can be given by the matrix 
$$ A \, = \, \left( \begin{array}{cccc} 
                   1 & 1 & 1 & 1 \\
                   0 & 0 & 1 & 1 \\
                   0 & 1 & 0 & 1 
                   \end{array} \right).$$
The kernel of $A$ is generated by $(1,-1, -1, 1)$, and 
the only two projection points are $(\frac{1}{2}, 0, 0, \frac{1}{2})$ and 
$(0, \frac{1}{2}, \frac{1}{2}, 0)$ with the MLE $q = (\frac{1}{4}, \frac{1}{4}, \frac{1}{4}, \frac{1}{4})$. 
Since the information divergence from both projection points is $\log(2)$ they are both global maximizers. 
\end{example}

We finish this section with a result that will be useful later. 
\begin{theorem} \cite[Lemma 3.2]{MRA11} \label{thm:block-matrix}
Let $A_1, A_2, \ldots, A_k$ be $d_i \times n_i$ matrices, $i = 1, \ldots, k$
with nonnegative entries and with the corresponding all ones vector as their first row. Let 
 $$A = \left( \begin{array}{cccc} A_1 & 0 & \cdots & 0\\ 0 & A_2 & \cdots & 0 \\
 \vdots & \vdots & \ddots & \vdots \\
 0 & 0 & \cdots & A_k \end{array} \right).$$ 
Then $D(\M_A) = \max \{D(\M_{A_1}), \ldots, D(\M_{A_k}) \}$. 
\end{theorem}
\begin{proof}
Let $n = \sum_{i=1}^k n_i$ and $d = \sum_{i=1}^k d_i$.
The toric variety $X_A$ as an affine variety
is $X_{A_1} \times \cdots \times X_{A_k}$ and the defining
toric ideal is $I_A = I_{A_1} + \cdots + I_{A_k}$.
Without loss of generality we can assume that $D(\M_{A_1})$ attains the maximum among the maximum information divergences 
for $\M_{A_1}, \ldots, \M_{A_k}$. Let $p^{(1)} \in \Delta_{n_1-1}$ be a global maximizer with the associated
MLE $q^{(1)}$. Setting $p = (p^{(1)}, 0, \ldots, 0)$
and $q = (q^{(1)}, 0, \ldots, 0)$, we get $p \in \Delta_{n-1}$ and $q \in \M_A = \Delta_{n-1} \cap X_A$. Since $D_{\M_A}(p \parallel q) = D_{\M_{A_1}}(p^{(1)} \parallel q^{(1)})$ we conclude that $D_{\M_A} \geq D_{\M_{A_1}}$. Conversely, 
let $p=(p^{(1)}, \ldots, p^{(k)})$ be a global maximizer
of $D_{\M_A}$ with the MLE $q = (q^{(1)}, \ldots, q^{(k)})$.
Set $p^{(i)}_+ = \sum_{j=1}^{n_i} p^{(i)}_j$ and
$q^{(i)}_+ = \sum_{j=1}^{n_i} q^{(i)}_j$. Note that 
$A_i p^{(i)} = A_i q^{(i)}$, so $p^{(i)}_+ = q^{(i)}_+$,
and $q^{(i)} \in X_{A_i}$. Moreover $\sum_{i=1}^k p^{(i)}_+ 
 =1$. Now let ${\tilde p}^{(i)} = 
\frac{1}{p^{(i)}_+}p^{(i)}$ and ${\tilde q}^{(i)} = 
\frac{1}{p^{(i)}_+}q^{(i)}$. We see that ${\tilde p}^{(i)} \in \Delta_{n_i-1}$, and ${\tilde q}^{(i)} \in \M_{A_i}$
is the MLE of ${\tilde p}^{(i)}$. Since 
$D({\tilde p}^{(i)} \parallel {\tilde q}^{(i)}) = 
\frac{1}{p^{(i)}_+} D( p^{(i)} \parallel q^{(i)})$ 
we conclude that $$ D(\M_A) = D(p \parallel q) = \sum_{i=1}^k D(p^{(i)} \parallel q^{(i)}) = \sum_{i=1}^k p^{(i)}_+ D({\tilde p}^{(i)} \parallel {\tilde q}^{(i)}) \leq \max\{D(\M_{A_1}), \ldots, D(\M_{A_k})\},$$
as desired.

\end{proof}

\section{The chamber complex and the algorithm} \label{sec:chamber-complex}
We devote this section to describing an algorithm to compute $D(\M_A)$ and the corresponding global maximizers for a toric model $\M_A$. We first introduce the chamber complex of $A$:
a polytopal complex $\C_A$ that is supported on the convex hull of (the columns of) $A$. This combinatorial object parametrizes all logarithmic Voronoi polytopes for the model $\M_A$. In particular, the finitely many faces (chambers) of $\C_A$
correspond to all possible combinatorial types of
these logarithmic Voronoi polytopes. It appears that, in order to locate all global maximizers of $D_{\M_A}$, one 
needs to examine the vertices of all logarithmic Voronoi polytopes. With the help of $\C_A$ we will reduce this task to examining vertices of each combinatorial type  where we essentially do an algebraic computation for each chamber in $\C_A$. For any omitted details in the definition and computation of $\C_A$ as well as its properties we refer to \cite[Chapter 5]{DRS10}.

Recall that $A$ is a $d \times n$ matrix with nonnegative integer entries and $\rank(A) = d$. We also assume that the first row of $A$ is the vector of all ones. This means that the convex hull of the columns of $A$, $\conv(A)$, is a polytope of dimension $d-1$ whose set of vertices is a subset of the columns of $A$. For a nonempty  $\sigma \subset [n]$ we let $A_\sigma = \{a_i \, : \, i \in \sigma\}$. When $|\sigma| = d$ and $A_\sigma$ is invertible, $\conv(A_\sigma)$ is a $(d-1)$-dimensional simplex. We will also use $\sigma$ to denote $\conv(A_\sigma)$. By Carath\'{e}odory's theorem \cite[Proposition 1.15]{Z95}, $\conv(A)$ is the union of all such simplices. 
%Except a subset of measure zero, if a point $b \in \conv(A)$ is contained in a $(d-1)$-simplex $\sigma$ then it is contained in the relative interior of $\sigma$. 

\begin{definition} \label{def:chamber-complex}
For $b \in \conv(A)$ let $C_b \, := \, \underset{\sigma \ni b}{\bigcap} \sigma$. The chamber complex of $A$ is
$$ \C_A \, := \, \{C_b \, :\, b \in \conv(A)\}.$$
\end{definition}
We note that $\C_A$ is a polytopal complex supported 
on $\conv(A)$, and each $C_b$ is a face of $\C_A$. Each such face of $\C_A$ is called a chamber. For every $b \in \conv(A)$ the set $Q_b = \{p \in \Delta_{n-1} \, : \, Ap = b\}$ is a logarithmic Voronoi polytope. 
The polytope $Q_b$ has the maximum dimension $n-d-1$
if and only if $b$ is in the relative interior of 
$\conv(A)$.
\begin{theorem} \label{thm:vertices}
Let $C$ be a chamber of the chamber complex $\C_A$.  Then for each $b$ in the relative interior of $C$, the vertices of $Q_b$ are in bijection with $\sigma \subset [n]$ such that $C$ is contained in the relative interior of $\conv(A_\sigma)$ where the columns of $A_\sigma$ are linearly independent. The support of the vertex corresponding to such $\sigma$ is 
precisely $\sigma$. More generally, 
each face $F$ of $Q_b$ is of the form
$F = Q_b \, \cap  \underset{i \not \in \supp(F)}{\bigcap} \{p_i = 0\}$. As $b$ varies in the 
relative interior of $C$, the support of each face of $Q_b$ as well as the combinatorial type of $Q_b$ does
not change. 
%$$\supp(F) = \underset{\supp(v) \subset \supp(F)}{\underset{\text{vertex} \, v:}{\bigcup}} \supp(v).$$
\end{theorem}
\begin{proof}
The polytope $Q_b$ is a polyhedron in standard form. 
Hence, $v \in \Delta_{n-1}$ is a vertex of $Q_b$ if and only if
$Av = b$ where there exists $\sigma \subset [n]$
such that the columns of $A_\sigma$ are linearly independent, and $i \not \in \sigma$ implies $v_i=0$; see
\cite[Theorem 2.4]{bertsimas-LPbook}. 
This is equivalent to $C \subset \conv(A_\sigma)$. The extra condition that $C$ is contained in the relative
interior of $\conv(A_\sigma)$ is equivalent to $\supp(v) = \sigma$. More generally, each face $F$ of $Q_b$ is
defined by some subset of coordinate hyperplanes $p_i =0$. Since $\supp(F)$ is the union of the supports
of all the vertices on $F$ we conclude that 
$F = Q_b \, \cap  \underset{i \not \in \supp(F)}{\bigcap} \{p_i = 0\}$. By the first part of this theorem, as $b$ varies in the relative interior of $C$, 
the support of each vertex does not change, and hence
the support of each face does not change. Since each face
is determined by the set of vertices contained in that face this implies that the face lattice of $Q_b$ is constant, i.e. every $Q_b$ has the same combinatorial type.
\end{proof}
\begin{example} \label{ex:chamber-pentagon}
Let 
$$ A = \left( \begin{array}{ccccc}
          1 & 1 & 1 & 1 & 1 \\
          0 & 1 & 2 & 3 & 2 \\
          1 & 0 & 0 & 1 & 2 \end{array} \right),$$
where we denote the columns of $A$ by $a, b, c, d$, and 
$e$. Here $\conv(A)$ is a pentagon which, together
with its chamber complex $\C_A$, can be 
seen in Figure \ref{fig:chamber-complex-pentagon}. 
This chamber complex consists of $10$ vertices, $20$ edges, and $11$ two-dimensional chambers. For some
chambers $C$ we have depicted the logarithmic 
Voronoi polytopes $Q_b$ where $b$ is in the relative interior of $C$. For instance, the horizontal
(red) edge of the pentagonal chamber supports  logarithmic Voronoi polytopes that are quadrangles. The supports of their vertices are $\{a,d\}, \{a,c,e\}, \{b,c,e\}$, and $\{b,d,e\}$ because $C$ is contained
in the relative interiors of $\conv(A_{a,d})$, 
$\conv(A_{a,c,e})$, $\conv(A_{b,c,e})$, and $\conv(A_{b,d,e})$.  
\end{example}

\begin{center}
    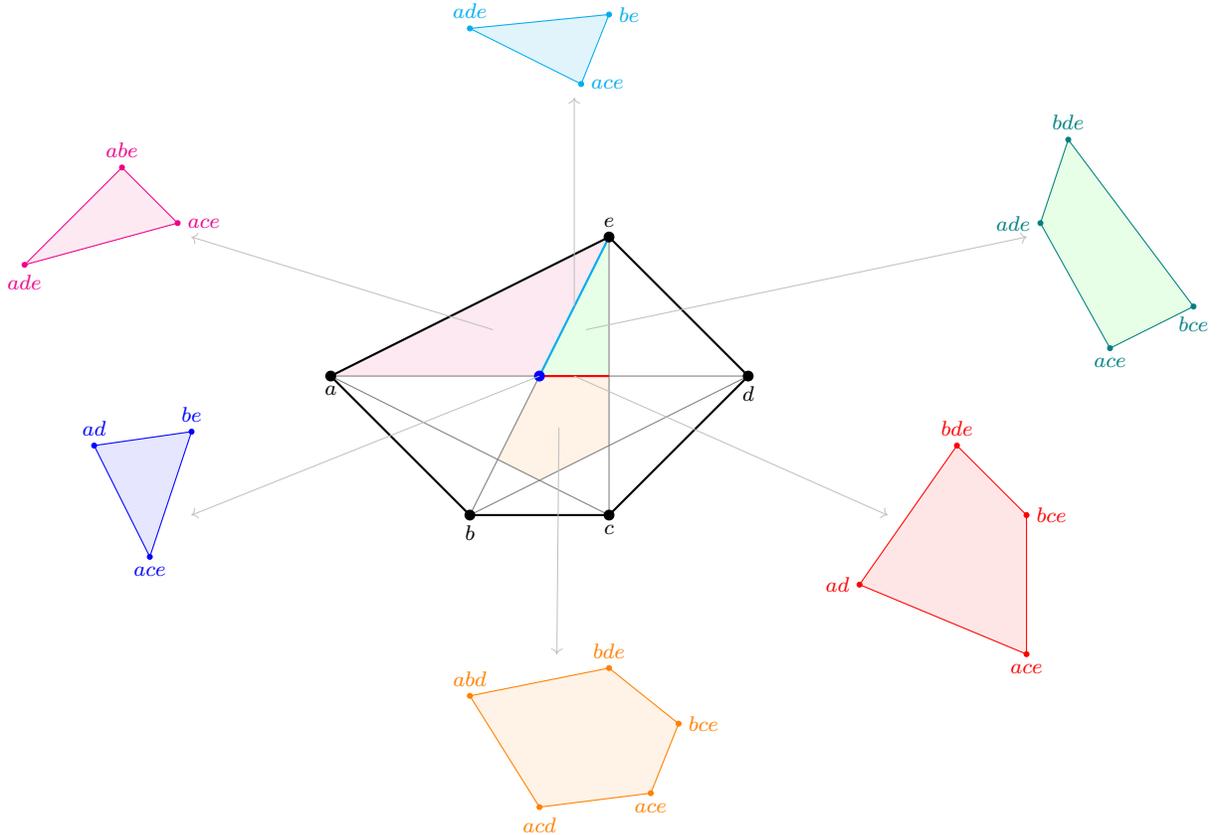
\begin{figure}[H]
        \centering
\begin{tikzpicture}[scale=1.85]

\draw[color=magenta!10,fill=magenta!10,thin] (0,1) -- (2,2) -- (3/2,1);
\draw[color=green!10,fill=green!10,thin] (2,2) -- (3/2,1) -- (2,1);
\draw[color=orange!10,fill=orange!10,thin]  (2, 1/2) -- (2, 1) -- (3/2, 1) --(6/5, 2/5) -- (3/2, 1/4);

%\draw[->] (a) -- (b);

\draw[black, thick] (0,1) -- (1,0);
\draw[black, thick] (1,0) -- (2,0);
\draw[black, thick] (2,0) -- (3,1);
\draw[black, thick] (3,1) -- (2,2);
\draw[black, thick] (0,1) -- (2,2);

\draw[gray, thin] (0,1) -- (2,0);
\draw[gray, thin] (0,1) -- (3,1);
\draw[gray, thin] (3,1) -- (1,0);
\draw[gray, thin] (2,2) -- (1,0);
\draw[gray, thin] (2,2) -- (2,0);

\filldraw[black] (0,1) circle (1pt) node[anchor=north] {\scriptsize{$a$}};;
\filldraw[black] (1,0) circle (1pt) node[anchor=north] {\scriptsize{$b$}};
\filldraw[black] (2,0) circle (1pt) node[anchor=north] {\scriptsize{$c$}};
\filldraw[black] (3,1) circle (1pt) node[anchor=north] {\scriptsize{$d$}};

\draw[color=red!100,thick] (3/2,1) -- (2,1);
\draw[color=cyan!100,thick] (2,2) -- (3/2,1);

\filldraw[black] (2,2) circle (1pt) node[anchor=south] {\scriptsize{$e$}};

\filldraw[blue] (3/2,1) circle (1pt);

\draw[->, color=gray!50, very thin] 
(7/6, 4/3) -- (-1,2);
\draw[->, color=gray!50, very thin] 
(11/6, 4/3)-- (5,2);
\draw[->, color=gray!50, very thin] 
(7/4, 3/2)-- (7/4,3);
\draw[->, color=gray!50, very thin] 
(7/4, 1)-- (4,0);
\draw[->, color=gray!50, very thin]
(3/2, 1)-- (-1,0);
\draw[->, color=gray!50, very thin]
(41/25, 63/100)-- (3.25/2,-1);

\draw[color=magenta,fill=magenta!10,thin] (-1.1,2.1) -- (-1.5,2.5) -- (-2.2,1.8) --(-1.1,2.1);
\filldraw[magenta] (-1.1,2.1) circle (0.5pt) node[anchor=west] {\scriptsize{$ace$}};
\filldraw[magenta] (-1.5,2.5) circle (0.5pt) node[anchor=south] {\scriptsize{$abe$}};
\filldraw[magenta] (-2.2,1.8) circle (0.5pt) node[anchor=north] {\scriptsize{$ade$}};

\draw[color=teal,fill=green!10,thin] (5.1,2.1) -- (5.3,2.7) -- (6.2,1.5) -- (5.6,1.2) -- (5.1,2.1);
\filldraw[teal] (5.1,2.1) circle (0.5pt) node[anchor=east] {\scriptsize{$ade$}};
\filldraw[teal] (5.3,2.7) circle (0.5pt) node[anchor=south] {\scriptsize{$bde$}};
\filldraw[teal] (6.2,1.5) circle (0.5pt) node[anchor=north] {\scriptsize{$bce$}};
\filldraw[teal] (5.6,1.2) circle (0.5pt) node[anchor=north] {\scriptsize{$ace$}};

\draw[color=cyan,fill=cyan!10,thin] (1.8,3.1) -- (2,3.6) -- (1,3.5) -- (1.8,3.1);
\filldraw[cyan] (1.8,3.1) circle (0.5pt) node[anchor=west] {\scriptsize{$ace$}};
\filldraw[cyan] (2,3.6) circle (0.5pt) node[anchor=west] {\scriptsize{$be$}};
\filldraw[cyan] (1,3.5) circle (0.5pt) node[anchor=south] {\scriptsize{$ade$}};

\draw[color=orange,fill=orange!10,thin] (2,-1.1) -- (2.5,-1.5) -- (2.3,-2) -- (1.5,-2.1) -- (1,-1.3) --(2,-1.1);
\filldraw[orange] (2,-1.1) circle (0.5pt) node[anchor=south] {\scriptsize{$bde$}};
\filldraw[orange] (2.5,-1.5) circle (0.5pt) node[anchor=west] {\scriptsize{$bce$}};
\filldraw[orange] (2.3,-2) circle (0.5pt) node[anchor=north] {\scriptsize{$ace$}};
\filldraw[orange] (1.5,-2.1) circle (0.5pt) node[anchor=north] {\scriptsize{$acd$}};
\filldraw[orange] (1,-1.3) circle (0.5pt) node[anchor=south] {\scriptsize{$abd$}};

\draw[color=blue,fill=blue!10,thin] (-1.3,-0.3) -- (-1.7,0.5) -- (-1,0.6) -- (-1.3,-0.3);
\filldraw[blue] (-1.3,-0.3) circle (0.5pt) node[anchor=north] {\scriptsize{$ace$}};
\filldraw[blue] (-1.7,0.5) circle (0.5pt) node[anchor=south] {\scriptsize{$ad$}};
\filldraw[blue] (-1,0.6) circle (0.5pt) node[anchor=south] {\scriptsize{$be$}};

\draw[color=red,fill=red!10,thin] (3.8,-0.5) -- (4.5,0.5) -- (5,0) -- (5,-1) -- (3.8,-0.5);
\filldraw[red] (3.8,-0.5) circle (0.5pt) node[anchor=east] {\scriptsize{$ad$}};
\filldraw[red] (4.5,0.5) circle (0.5pt) node[anchor=south] {\scriptsize{$bde$}};
\filldraw[red] (5,0) circle (0.5pt) node[anchor=west] {\scriptsize{$bce$}};
\filldraw[red] (5,-1) circle (0.5pt) node[anchor=north] {\scriptsize{$ace$}};

\end{tikzpicture} \;\;\;\;\;\; 
        \caption{The chamber complex of a pentagon.}
        \label{fig:chamber-complex-pentagon}
\end{figure}
\end{center}

\begin{remark}
Although each chamber of $\C_A$ gives rise to logarithmic Voronoi polytopes of $\M_A$ that have
the same combinatorial type, different chambers might
yield identical combinatorial types. For instance, 
in Example \ref{ex:chamber-pentagon} we see that there 
are multiple chambers that support logarithmic Voronoi
polytopes that are triangles or quadrangles. In fact,
$\C_A$ parametrizes these polytopes according to a finer invariant, namely, the normal fan of each polytope. 
We will not directly need this finer differentiation, though
we will use the fact from Theorem \ref{thm:vertices} that the supports of the faces of $Q_b$ given by $b$ in 
a fixed chamber are constant. 
\end{remark}
\begin{remark} \label{rmk:compute-chambers}
In the algorithm we present, first we have to compute
the chamber complex $\C_A$. Using Definition \ref{def:chamber-complex} for this computation is highly inefficient. Here is an outline for a more efficient way. First, one computes a Gale transform $B$ of $A$ where $B$ is a $(n-d)\times n$ matrix whose rows form
a basis for the kernel of $A$. Then the secondary 
fan $\Pi_B$ of $B$ is computed. This is a complete
fan in $\RR^n$ in which each cone consists of weight
vectors that induce the same regular subdivision of 
the vector configuration given by the $n$ columns of $B$.
The cones of the secondary fan $\Pi_B$ are in bijection with the chambers of $\C_A$. More precisely, if 
$u_1, \ldots, u_k$ are the generators of a cone in $\Pi_B$ the corresponding chamber in $\C_A$ is the convex
hull of $Au_1, \ldots, Au_k$. The details can be found in \cite[Section 5.4]{DRS10}; in particular, for the claimed bijection see Theorem 5.4.5 in the same reference. We used {\em Gfan} \cite{gfan} to compute $\Pi_B$ which can also be accessed via {\em Macaulay 2} \cite{M2}.
\end{remark}

\begin{example} \label{ex:2x3x3chamber}
As the matrix $A$ gets larger, all of these computations become challenging. To give an idea, we consider the toric model $\M_A$ that is the independence model of a binary and two ternary random variables. It is a $5$-dimensional model in $\Delta_{17}$. The $f$-vector of 
the $5$-dimensional polytope $\conv(A)$ is
$(18, 45, 48, 27, 8)$, i.e., this polytope has $18$ vertices, $45$ edges, etc. The chamber complex $\C_A$
that was computed via the methods outlined in Remark 
\ref{rmk:compute-chambers} has the $f$-vector
$$(3503407, \, 33084756, \, 105341820, \, 151227738, \, 100828884, \, 25361616).$$
The computation took about two days on a standard laptop, and it could only be done after taking into account the symmetries of $\conv(A)$. We note that, luckily, this computation needs to be done only once, and
once $\C_A$ is computed, its chambers have to be processed as we will explain in our algorithm. This processing can be shortened by considering the symmetries
of the chamber complex (if there are any) as well as 
by using a few simple observations on the structure 
of the supports of the vertices of the logarithmic Voronoi polytopes. We will outline these ideas below.
\end{example}

According to Theorem \ref{thm:critical-projection}, given
a logarithmic Voronoi polytope $Q_b$ where $b \in \conv(A)$, we need to identify complementary vertices of $Q_b$ and decide whether any of these vertices are projection points. These, in turn, are potential local and global maximizers of $D_{\M_A}$. The following proposition gives a way to decide whether a complementary vertex is a projection point. 
\begin{prop} \label{prop:check-projection-point}
Let $v$ be a complementary vertex of the logarithmic
Voronoi polytope $Q_b$ of a toric model $\M_A$ with the complementary face $F$. Let $\L_{v,F}$ be the collection 
of the lines passing through $v$ and each point on $F$.
Then $v$ is a projection point if and only if $\L_{v,F}$
intersects $\M_A$.
\end{prop}
\begin{proof}
By Birch's theorem, $Q_b$ intersects $\M_A$ in a single
point, namely, the MLE $q$ of any point $p$ in $Q_b$. 
The vertex $v$ is a projection point if and only if one of the lines in $\L_{v,F}$ passes through $q$. This happens  if and only if $\L_{v,F}$ intersects $\M_A$ in the only possible point $q$.
\end{proof}
In light of Proposition \ref{prop:check-projection-point}, to check whether a complementary vertex $v$ of $Q_b$ is a projection point reduces to an algebraic computation. Let $F$ be the complementary face 
of dimension $k$. Then $\overline{\L_{v,F}}$, the Zariski closure of $\L_{v,F}$, is an affine 
subspace of dimension $k+1$ whose defining 
equations can easily be computed. For instance, if
$v_1, \ldots, v_{k+1}$ are vertices of $F$
that are affinely independent, then $\overline{\L_{v,F}}$
is the image of the map 
$$ (s, t_1, \ldots, t_{k+1}) \mapsto sv + (1-s)(t_1v_1 + \cdots + t_{k+1}v_{k+1})$$
 where $t_1 + \ldots + t_{k+1} =1$. To intersect 
$\overline{\L_{v,F}}$ with $\M_A$ we use the equations of $\overline{\L_{v,F}}$
and the binomial equations defining the toric variety
$X_A$. Since $\overline{\L_{v,F}}$ is contained in the affine span
of $Q_b$, and since the latter affine subspace intersects
$X_A$ in finitely many complex points (see Definition \ref{def:ml-degree}), $\overline{\L_{v,F}}$ intersects $X_A$ also in finitely many points. They can be computed using a numerical algebraic geometry
software such as {\em Bertini} \cite{Bertini} or {\em HomotopyContinuation.jl} \cite{HomotopyContinuation.jl}.
Finally, one checks whether this finite set contains
a point with positive coordinates. 

\begin{example} \label{ex:pentagon2}
We use Example \ref{ex:chamber-pentagon}. The point
$b= (1, 7/4, 1)$ is the midpoint of the horizontal
(red) edge of the pentagonal chamber. The vertex 
$v = (5/12, 0, 0, 7/12, 0)$ of $Q_b$ is complementary to 
another vertex $v_1 = (0, 1/4, 1/4, 0, 1/2)$. 
The toric variety $X_A$ is defined by the equations
$$ p_2^2p_4^2-p_3^3p_5 \, = \,  p_1p_3^3-p_2^3p_4 \, = \,  p_1p_4-p_2p_5 \, = \, 0.$$
The affine subspace spanned by $v$ and $v_1$ is just 
a line defined by 
$$ 12p_4 + 14p_5 - 7 \, = \,  2p_3 - p_5 \, = \, 2p_2 - p_5 \, = \, 12p_1 + 10p_5 - 5 \, = \, 0.$$
The intersection of $X_A$ with $\overline{\L_{v,\{v_1\}}}$ is
empty. Hence, we conclude that $v$ is not a projection
point.
\end{example}
The above discussion describes a way of checking 
whether a complementary vertex of a {\it fixed} logarithmic Voronoi polytope $Q_b$ is a projection point. Next, we describe how to accomplish the same task
for a complementary vertex of $Q_b$ as $b$ varies 
in the interior of a fixed chamber $C$ in the chamber complex $\C_A$.  By Theorem \ref{thm:vertices} each
such $Q_b$ has the same combinatorial type and the support of any face of $Q_b$ stays constant. Now let $(v(b), F(b))$ be a pair of a complementary vertex
and its corresponding complementary face in $Q_b$ where $b$ is in the relative interior of a chamber $C$. 
Let $w_1, \ldots, w_m$ be the vertices of $C$. Then
$b = \sum_{i=1}^m r_iw_i$ where $\sum_{i=1}^m r_i = 1$
and $r_i \geq 0$ for all $i=1,\ldots, m$. This 
means that the coordinates of $v(b)$ and those of the
vertices $v_1(b), \ldots, v_z(b)$ of $F(b)$ are linear functions of $r_1, \ldots, r_m$. Next, we parametrize
a general point $w(b)$ on $F(b)$ via
$w(b) = \sum_{i=1}^z t_iv_i(b)$ where $\sum_{i=1}^z t_i =1$ and $t_i \geq 0$ for all $i=1,\ldots, z$. 
Finally, the line segment between $v(b)$ and $w(b)$
is parametrized by $s v(b) + (1-s)w(b)$ where 
$0 \leq s \leq 1$. The last expression gives points
in $\Delta_{n-1}$ where each coordinate is a polynomial
in the parameters $r_1, \ldots, r_m$, $t_1, \ldots, t_z$,
and $s$, and it defines the map 
$$\Psi_{v,F} \, : \, \Delta_{m-1} \times \Delta_{z-1} \times \Delta_1 \longrightarrow \Delta_{n-1}.$$
Proposition \ref{prop:check-projection-point} implies
that $v(b)$ is a projection point for some $b \in C$
if and only if the image of $\Psi_{v,F}$ intersects 
$\M_A$. Again, this boils down to an algebraic computation. We substitute the coordinates
of $s v(b) + (1-s)w(b)$ into the equations defining 
$X_A$, check whether this system of equations 
has solutions in $\CC^{m + z +1}$, and if there 
are any, compute $\mathrm{im} \Psi_{v,F} \cap \M_A$
by imposing the positivity constraints on the solution set. The resulting semi-algebraic set 
is then the feasible region over which  $D_{\M_A}$ can
be maximized to identify local maximizers with support
equal to the support of $v(b)$. Finally, we locate
the global maximizer(s) among these local maximizers
contributed by each chamber $C$ of the chamber
complex $\C_A$ that supports projection points. We summarize this in a high-level algorithm.

\noindent\textbf{Algorithm:}

\textbf{Input:} $A \in \NN^{d \times n}$ that defines a toric model $\M_A \subset \Delta_{n-1}$  of dimension $d-1$.\\
\textbf{Output:} All maximizers of $D_{\M_A}$.

\vspace{-1em}
\begin{enumerate}
\itemsep0em 
    \item Compute the equations of the toric variety $X_A$.
    \item Compute the chamber complex $\C_A$.
    \item For each chamber $C$ in $\C_A$ do:
    \begin{itemize}
    \item[a)] for any fixed $\hat{b}$ in the relative interior of $C$ compute the face lattice of $Q_{\hat{b}}$ and identify complementary vertex/face pairs $(v,F)$;
    \item[b)] for each $(v,F)$ do:
      \begin{itemize}
          \item[i.] compute the parametrization $\Psi_{v,F}$ and substitute it into the equations of $X_A$;
          \item[ii.] if the resulting algebraic set in $\CC^{m+z+1}$ is nonempty then 
          \begin{itemize}
          \item compute the semi-algebraic set $\mathrm{im} \Psi_{v,F} \cap \M_A$ by imposing
          positivity constraints on the parameters in 
          $\Psi_{v,F}$;
          \item find the maximizers $D_{C,v,F}$ of $D_{\M_A}$ over $\mathrm{im} \Psi_{v,F} \cap \M_A$.
        \end{itemize}
      \end{itemize}
    \end{itemize}
\item Identify global maximizers of $D_{\M_A}$ by
comparing all $D_{C,v,F}$.
\end{enumerate}
\begin{example}
We illustrate this algorithm using the toric model of Example \ref{ex:chamber-pentagon}. The equations of $X_A$
are the three polynomials computed in Example \ref{ex:pentagon2}. The chamber complex $\C_A$ is the polytopal complex in Figure \ref{fig:chamber-complex-pentagon}. The chambers which support complementary
vertices are the (relative interior of) the boundary
edges of the pentagonal chamber. Step 3 is executed
only for these chambers. For instance, the horizontal
edge is the convex hull of its vertices $(3/2,1)$ and
$(2,1)$, and the unique complementary vertex/face pair
$(v,F)$ is given by vertex $v$ with support $\{a,d\}$ and  the vertex $F$ with support $\{b,c,e\}$. 
We note that for such pair of complementary {\it vertices} $(v, \{w\})$ we do not need to consider 
the pair $(w,\{v\})$ in the next computation.
The parametrization $\Psi_{v,F}$ is 
given by
$$ (r_1, s) \mapsto \left(s(\frac{1}{6}r_1+\frac{1}{3}), \, (1-s)\frac{r_1}{2}, \, (1-s)\frac{1-r_1}{2}, \, s(-\frac{1}{6}r_1 + \frac{2}{3}), \, \frac{1-s}{2} \right),$$
where we are parametrizing $b$ on this edge
by $r_1(3/2, 1) + (1-r_1)(2,1)$. Substituting 
$\Psi_{v,F}$ into the equations of $X_A$ results
in 
\footnotesize
\begin{align*}
    s^2r_1^2+7s^2r_1-8s^2-18sr_1+9r_1 =  0 \hspace{30.5em}\\
    197s^4r_1-194s^4-1401s^3r_1-3sr_1^3+1014s^3+4398s^2r_1+246sr_1^2-2094s^2-5837sr_1-81r_1^2+2s+2349r_1 & =0\\
     885s^4-31312s^3r_1-294sr_1^3+32392s^3+179435s^2r_1+17016sr_1^2
    -117350s^2-295438sr_1-6165r_1^2+2560s&\\+129141r_1-591 &= 0.
\end{align*}

\normalsize
This is a zero-dimensional system that has $11$ solutions
which we have computed using {\it Bertini}. Four of these are complex and seven
are real. There is a unique real solution where 
$0 < r_1, s < 1$, namely
$$ r_1 = 0.4702953126494577 \mbox{ and }  s= 0.4106301713351522.$$
The corresponding KL-divergence at the vertex $v$ is $0.890062259952966$. At the vertex $w$, the divergence is $0.528701425022976$.
For each of the remaining four edges of this pentagonal chamber we also get a pair of projection vertices with corresponding KL-divergences equal to 
\begin{align*}
&0.729916767214609\text{ and }0.657681783609608\\
&0.736523721240758\text{ and }0.651574202843057\\
&0.927851227501820\text{ and }0.503192212618303\\
&0.856820834934792\text{ and }0.552532602066626.
\end{align*}
The global maximizer is the vertex
$$v = (0, 0.6722451790633609, 0, 0, 0.3277548209366391)$$ corresponding to the divergence value 0.927851227501820. It is a vertex of the logarithmic Voronoi polytope $Q_b$ where $b = (1.3277548209366392, 0.6555096418732783)$ lies on the edge of the pentagonal chamber contained in the line segment between $(1,0)$ and $(2,2)$.
\end{example}

The basic algorithm above can be improved on
many fronts. We will now present some ideas for such improvements.

For Step 1, one could replace the equations of $X_A$, which 
could be challenging to compute for large models, with $n-d$
equations corresponding to a basis of $\ker_\ZZ(A)$. Let $B$
be an $(n-d) \times n$ matrix whose rows $b_i$, $i=1, \ldots, n-d$ form such a basis. The lattice basis ideal 
$$I_B = \langle \prod_{j=1}^n p_j^{b_{ij}^+} - \prod_{j=1}^n p_j^{b_{ij}^-}, \,\,\ i=1,\ldots, n-d \rangle$$
where $b_i = b_i^+ - b_i^-$ with $b_i^+, b_i^- \geq 0$ and
$\supp(b_i^+) \cap \supp(b_i^-) = \emptyset$ defines a variety
$Y_B$ containing $X_A$. In fact, $Y_B$ is the union of 
$X_A$ together with varieties contained in various coordinate
subspaces defined by setting a subset of coordinates equal to zero (see \cite[Section 8.3]{S02} and \cite{HS00}). 
This means that $\M_A^{>0} = X_A \cap \Delta_{n-1}^\circ$ 
is equal to $Y_B \cap \Delta_{n-1}^\circ$. This is what is 
ultimately needed in Step 3.b.ii. 

For Step 2, Example \ref{ex:2x3x3chamber} illustrated that computing $\C_A$ might be out of reach due to the combinatorial explosion in the number of chambers. However, one does
not need to compute $\C_A$ all at once. It can be computed
one chamber at a time. This is how a software like {\it Gfan} 
\cite{gfan} internally computes $\C_A$ based on reverse
search enumeration \cite{AF96}. In this case, Step 3 can 
be executed as chambers get computed. 

In Step 3, not all chambers need to be considered. For instance, any chamber that is contained in the boundary of 
$\conv(A)$ can be skipped: if $b$ is in such a chamber, the 
logarithmic Voronoi polytope $Q_b$ is contained in the boundary of $\Delta_{n-1}$. Such $Q_b$ does not contribute 
global maximizers of $D_{\M_A}$. There are also ways to eliminate chambers since they cannot contain complementary vertices. We present a few ways this can be done. 

\begin{prop} \label{prop:chamber-boundary}
Suppose $\conv(A)$ is a simplicial polytope where
each column of $A$ is a vertex. Let $C$ be a chamber that intersects the boundary  as well as the interior of $\conv(A)$. Then for any $b$ that is in the relative interior of $C$, the logarithmic Voronoi polytope $Q_b$ does not contain complementary vertices. 
\end{prop}
\begin{proof}
The intersection of $C$ with the boundary of $\conv(A)$ is a simplex spanned
by a subset of columns of $A$, say $A_{i_1}, \ldots, A_{i_k}$.
Then the support of every vertex of $Q_b$ contains
$\{i_1, \ldots, i_k\}$. This disallows the existence
of complementary vertices. 
\end{proof}
Note, for instance, in our running Example \ref{ex:chamber-pentagon}, it is enough to consider the pentagonal chamber and
its faces by the above proposition. In fact, the interior
of this chamber does not have to be considered either for
the following reason. 
\begin{prop}\label{prop:support-reasons-dimension}
Let $C$ be a chamber of  dimension $k$ where $k+1 > n/2$. 
Then for any $b$ that is in the relative interior of $C$, the logarithmic Voronoi polytope $Q_b$ does not contain complementary vertices. 
\end{prop}
\begin{proof}
    By Theorem \ref{thm:vertices}, each of the vertices of $Q_b$ has support of size at least $k+1$. If a vertex $v$ of $Q_b$ is complementary there must exist a vertex $w$ such that $\supp(v)\cap \supp(w)=\varnothing$. Such two vertices can only exist when $2(k+1) \leq n$. 
\end{proof}
\begin{prop} \label{prop:support-on-the-boundary}
 If $(v,F)$ is a pair of a complementary vertex and
 its complementary face $F$ where both $v$ and $F$ are 
 contained in the same facet $F'$ of $Q_b$, then $v$ cannot be a projection point.
\end{prop}
\begin{proof}
  The line segments from $v$ to the points in $F$ are entirely contained in $F'$ which is in the boundary of $\Delta_{n-1}$. Then $v$ cannot be a projection point
  since no such line segment can intersect the toric model $\M_A$ in the interior of $\Delta_{n-1}$.
\end{proof}
Again, we note that, Proposition \ref{prop:support-on-the-boundary} rules out the zero-dimensional chambers that are 
the vertices of the pentagonal chamber in Example \ref{ex:chamber-pentagon} since they give rise to complementary pairs $(v,F)$ lying in the same facet of 
their logarithmic Voronoi polytope. 
\begin{prop}\label{prop:face-containing-face-out}
Let $C$ be a chamber in the chamber complex $\C_A$. If no two vertices of the logarithmic Voronoi polytope $Q_b$ corresponding to points $b$ in the relative interior of $C$ have disjoint supports, then the same is true for any chamber $C'$ containing $C$.
\end{prop}
\begin{proof}
The supports of vertices of $Q_{b'}$ where $b'$ is in the 
relative interior of $C'$ are in bijection with $\sigma'$
where columns of $A_{\sigma'}$ are affinely independent
and the relative interior of $\conv(A_{\sigma'})$ contains
the relative interior of $C'$. 
Since $C$ is a face of $C'$, for any $\sigma$ such 
that the columns of $A_\sigma$ are affinely independent
and the relative interior of $\conv(A_\sigma)$ contains
the relative interior of $C$, there is (possibly multiple) $\sigma' \supset \sigma$ as above. Hence if no two vertices
of $Q_b$ have disjoint supports, the same is true for $Q_{b'}$.
\end{proof}

\begin{cor}
Let $A \in \NN^{3 \times 5}$ such that $\conv(A)$ is  a planar pentagon. If a logarithmic Voronoi polytope $Q_b$ 
contains a projection point then $b$ is in the interior of  
an edge of the pentagonal chamber.  Moreover, each such 
edge contributes either finitely many projection points or
for every $b$ on the edge, $Q_b$ has a projection point. 
\end{cor}
\begin{proof}
Propositions \ref{prop:chamber-boundary}, \ref{prop:support-reasons-dimension}, and \ref{prop:support-on-the-boundary} 
imply the first statement. Any logarithmic 
Voronoi polytope $Q_b$ where $b$ is on an edge
of the pentagonal chamber has a pair of complementary vertices $(v,w)$. The 
Zariski closure of the image of $\Psi_{v,w} \, : \,  \Delta_1 \times \Delta_1 \longrightarrow \Delta_4$ in $\PP_\CC^4$ 
is a two-dimensional irreducible surface. Since $X_A$ is also  two-dimensional and irreducible, and it is never equal to the former Zariski closure, their intersection has either finitely many points (this is the generic case) or it is an algebraic curve. This means 
that $\mathrm{im} \Psi_{v,w} \cap \M_A$ has either
finitely many points or contains the positive real part of an algebraic curve. In the second case, the projection of the preimage of this positive real part under $\Psi_{v,w}$ to the first $\Delta_1$ in the domain of $\Psi_{v,w}$ must 
be all of $\Delta_1$. Hence, for every $b$ on this edge,
$Q_b$ has a projection point.
\end{proof}   

Our final remark about the algorithm concerns the step 
where $D_{\M_A}$ needs to be maximized over the semi-algebraic
set $\mathrm{im} \Psi_{v,F} \cap \M_A$. Of course, this is 
a challenging step. However, generically one expects 
this set to be finite. In that case, numerical algebraic 
geometry tools perform well to compute each point in this
finite intersection. Another relatively easier case is when
the maximum likelihood degree of $X_A$ is one; see Definition \ref{def:ml-degree}. There are two advantages in this case. First,
the intersection of $\mathrm{im} \Psi_{v,F}$ with $X_A$ is
guaranteed to be in $\Delta_{n-1}$ since the affine span of each logarithmic Voronoi polytope $Q_b$ intersects $X_A$ in exactly one point, 
namely the unique maximum likelihood estimator $q(b)$ in $Q_b$. 
Second, $q(b)$ is a rational function of $v(b)$ -- an equivalent
condition for an algebraic statistical model to have maximum
likelihood degree equal to one. In other words, both $v(b)$ and $q(b)$
are rational functions of the parameters $(r_1, \ldots, r_m) \in \Delta_{m-1}$. In turn, $D_{\M_A}$ restricted 
to the potential projection points $v(b)$ is a greatly 
simplified function of the same parameters. Now, one needs
to optimize $D_{\M_A}(r_1,\ldots, r_m)$ over $\Delta_{m-1}$.

\begin{example}[Independence model $2\times 3$] \label{ex:2x3-chamber}
Consider the independence model of two random variables, binary $X$ and ternary $Y$. Similar to Example \ref{ex:two-binary-RV}, this is a $3$-dimensional toric model inside $\Delta_5$ given by the matrix
$$ A \, = \, \left( \begin{array}{cccccc} 
                   1 & 1 & 1 & 1 & 1 & 1 \\
                   0 & 0 & 0 & 1 & 1 & 1 \\
                   0 & 1 & 0 & 0 & 1 & 0 \\
                   0 & 0 & 1 & 0 & 0 & 1 
                   \end{array} \right).$$
The polytope $\conv(A)$ is a 3-dimensional polytope with six vertices that is highly symmetrical due to the action of the group $S_2 \times S_3$ on the states of $X$ and $Y$. This, in turn, induces a partition of the elements in the chamber complex $\C_A$ into symmetry classes. This way, 18 full-dimensional chambers are split into 5 classes, 44 ridges are split into 7 classes, 36 edges are split into 6 classes, and 11 vertices are split into 3 classes. Figure \ref{fig:binary-and-ternary} demonstrates this division of full-dimensional chambers: any chambers that share a color are in the same symmetry class. The red middle chamber is a bipyramid with a triangular base and is the only one in its class.
\end{example}

\begin{figure}[H]
    \centering
    \includegraphics[width=0.32\textwidth]{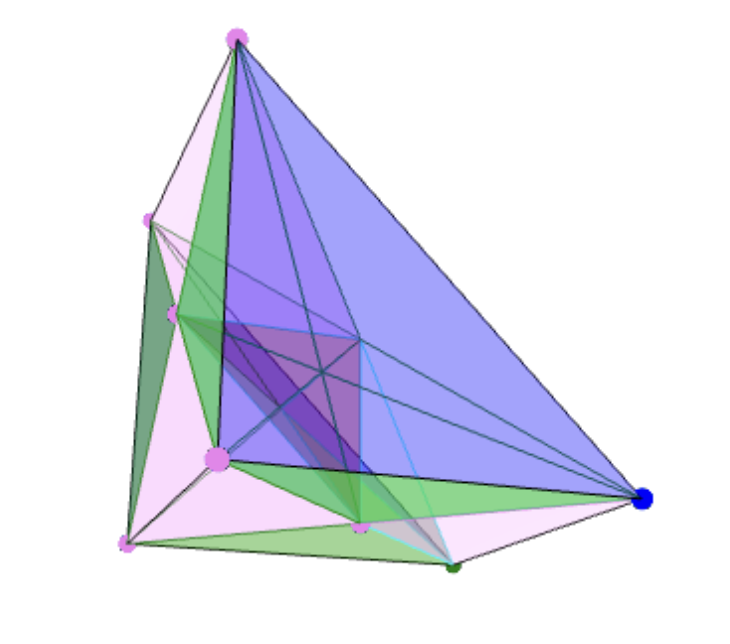}
    \includegraphics[width=0.25\textwidth]{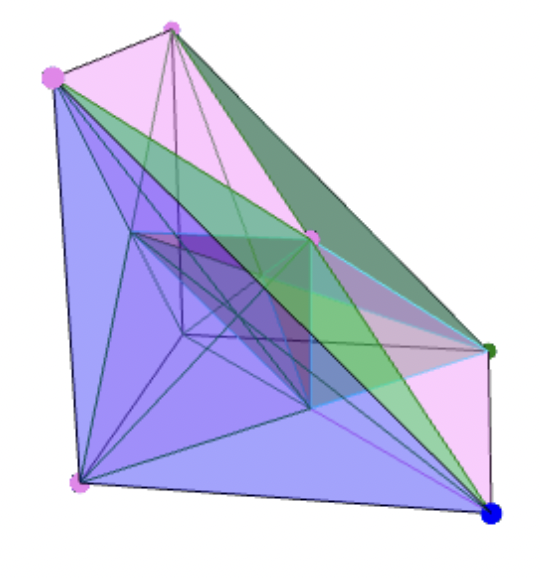}
    \includegraphics[width=0.3\textwidth]{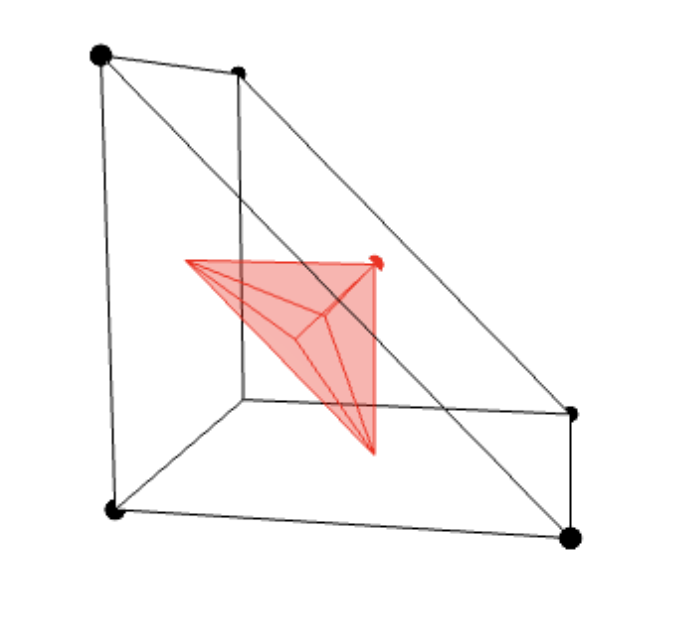}
    \caption{Chamber complex of the $2\times 3$ independence model (left and middle) and the middle chamber (right).}
    \label{fig:binary-and-ternary}
\end{figure}
To run the algorithm, note that 9 out of the 11 vertices are on the boundary of $\conv(A)$, and hence do not contribute any projection vertices by Proposition \ref{prop:support-on-the-boundary}. Call the two interior vertices $b_1$ and $b_2$. They are both in the same symmetry class and lie on the middle red full-dimensional chamber. The logarithmic Voronoi polytope at a point corresponding to $b_1$ is a triangle with no complementary vertices, and the same is true of $b_2$ by symmetry. Hence, no vertices of the chamber complex will contribute any projection points. Next, out of 36 edges 21 are on the boundary. Moreover, 6 of the remaining edges contain the vertex $b_1$ and by symmetry another 6 contain the vertex $b_2$, so we do not need to check these edges by Proposition \ref{prop:face-containing-face-out}. This leaves us with three edges $e_1$, $e_2$, and $e_3$ on the base of the red bipyramid. We will treat them in the next paragraph. Out of 44 ridges, 14 are on the boundary, 12 contain vertex $b_1$, and another 12 contain vertex $b_2$. The remaining 6 are in the same symmetry class. Logarithmic Voronoi polytopes corresponding to these are quadrilaterals with supports like $\{1234, 1345, 1246, 156\}$ that contain no complementary vertices. Hence, none of the ridges will contribute projection points. Finally, none of the three-dimensional chambers will contribute any projection points by Proposition $\ref{prop:support-reasons-dimension}$.

Hence, we only need to run step 3 of our algorithm on the edges $e_1, e_2,$ and $e_3$. By symmetry, it suffices to run it on $e_1$ only. A point $b$ on this edge can be parametrized as $r_1(1/2, 1/2, 0)+(1-r_1)(1/2, 0, 1/2)$. The only vertex-face pair we need to consider is the pair of complementary vertices $(v,\{w\})$, where $v=1/2(1, 0, 0, 0, r_1, 1-r_1)$ and $w=1/2(0, r_1, 1-r_1, 1, 0, 0)$. The parametrization $\Psi_{v,\{w\}}$ of the line between them gives rise to the single equation $(s - 1)^2 - s^2=0$. Therefore $s=1/2$, while $r_1$ is a free variable between 0 and 1. Upon substituting $s=1/2$ into $D(v,\text{im}(\Psi_{v,\{w\}}))$, we get the constant value $\log 2$. Therefore, the divergence at \textit{every} point $b$ of the edge $e_1$ is $\log 2$, attained at the two vertices of the logarithmic Voronoi polytope $v$ and $w$. By symmetry, the same is true of $e_2$ and $e_3$. We conclude that the maximum divergence from this model is $\log 2$ and there are infinitely many maximizers which we completely characterized above. These maximizers were also studied and visualized in \cite{AK06}.

\section{Logarithmic Voronoi polytopes of reducible hierarchical log-linear models}\label{sec:decomposable}

This section is devoted to logarithmic Voronoi polytopes 
of toric models that are known as reducible hierarchical 
log-linear models \cite{DrtonSturmfelsSullivant2009LecturesOnAlgebraicStatistics, HS02, Lau96}. Besides giving one 
structural result about these polytopes, we will also prove results relating the maximum information divergence to such 
models with those that are obtained by certain marginalizations. For similar work we refer the reader to 
\cite{Mat09}.

A \textit{simplicial complex} is a set $\Gamma\subseteq 2^{[m]}$ such that if $F\in\Gamma$ and $S\subseteq F$, then $S\in\Gamma$. The elements of $\Gamma$ are called \textit{faces}. We refer to inclusion-maximal faces of $\Gamma$ as \textit{facets}. It is sufficient to list the facets to describe a simplicial complex.  For example, $\Gamma=[12][13][23]$ will denote the simplicial complex $\Gamma=\{\varnothing, \{1\}, \{2\}, \{3\},\{1, 2\}, \{1, 3\},\{2,3\}\}$.

Let $X_1,\ldots, X_m$ be discrete random variables. For each $i\in[m]$, assume that $X_i$ has the state space $[d_i]$ for some $d_i\in\NN$. Let $\R=\prod_{i=1}^m[d_i]$ be the state space of the random vector $(X_1,\ldots, X_m)$. For each $i=(i_1, \ldots, i_m)\in \R$ and $F=\{f_1,f_2,\ldots\}\subseteq[m]$, we will denote $i_F=(i_{f_1},i_{f_2},\ldots)$. Moreover, each such subset $F\subseteq[m]$ gives rise to the random vector $X_F=(X_f)_{f\in F}$ with the state space $\R_F=\prod_{f\in F}[d_f]$.

\begin{defn}
Let $\Gamma\subseteq 2^{[m]}$ be a simplicial complex and let $d_1,\ldots, d_m\in\NN$. For each facet $F\in\Gamma$, introduce $|\R_F|$ parameters $\theta^{(F)}_{i_F}$, one for each $i_F\in\R_F$. The \textit{hierarchical log-linear model} associated with $\Gamma$ and $\ddb = (d_1, \ldots, d_m)$ is defined to be 
$$\M_{\Gamma, \ddb}=\left\{p\in\Delta_{|\R|-1}:p_i=\frac{1}{Z(\theta)}\prod_{F\in\text{facets}(\Gamma)}\theta_{i_F}^{(F)}\text{ for all }i\in\R\right\}$$
where $Z(\theta)$ is the normalizing constant defined as
$$Z(\theta):=\sum_{i\in\R}\prod_{F\in\text{facets}(\Gamma)}\theta_{i_F}^{{F}}.$$
\end{defn}

If $u\in\NN^{|\R|}$ is a $d_1\times\cdots\times d_m$ contingency table containing data for the random vector $(X_1,\ldots, X_m)$  and $F=\{f_1,f_2,\ldots\}\subseteq[m]$, let $u_F$ denote the $d_{f_1}\times d_{f_2}\times\cdots$ table with $(u_F)_{i_F}=\sum_{j\in\R_{[m]\setminus F}}u_{i_F,j}$. Such table $u_F$ is called the \textit{$F$-marginal} of $u$. For simplicity, we will denote the simplex in which $\M_{\Gamma, \ddb}$ lives by $\Delta_{\Gamma, \ddb}$.

\begin{prop}\cite[Prop. 1.2.9]{DrtonSturmfelsSullivant2009LecturesOnAlgebraicStatistics} \label{prop:hierarchical=toric}
Hierarchical log-linear models are toric models. For any simplicial complex $\Gamma \subset 2^{[m]}$ and positive
integers $\ddb = (d_1, \ldots, d_m)$,  the model $\M_{\Gamma, \ddb}$ is realized by the 0/1 matrix $A_{\Gamma, \ddb}$ representing the marginalization map $$ \varphi(u) = (u_{F_1},u_{F_2},\ldots)$$ where $F_1,F_2,\ldots$ are the facets of $\Gamma$. In other words, $\M_{\Gamma, \ddb} =\M_{A_{\Gamma, \ddb}}$.
\end{prop}
Here we wish to point out that for any point $q \in \Delta_{\Gamma, \ddb}$ (in particular, for $q \in \M_{\Gamma, \ddb}$) the logarithmic Voronoi polytope 
$Q_q^\Gamma$ consists of all $p \in \Delta_{\Gamma, \ddb}$
such that $\varphi(p) = \varphi(q)$. 
\begin{definition} \label{def:reducible}
A simplicial complex $\Gamma$ on $[m]$ is called \textit{reducible} with decomposition $(\Gamma_1,S,\Gamma_2)$ if there exist sub-complexes $\Gamma_1$, $\Gamma_2$ of $\Gamma$ and a subset $S\subseteq[m]$ such that $\Gamma=\Gamma_1\cup\Gamma_2$ and $\Gamma_1\cap\Gamma_2=2^S$. We say $\Gamma$ is \textit{decomposable} if it is reducible and each of the $\Gamma_1,\Gamma_2$ is either decomposable or a simplex. A hierarchical log-linear model associated to a reducible (decomposable) simplicial complex is called reducible (decomposable).
\end{definition}

\subsection{Decomposition theory of logarithmic Voronoi polytopes}

Let $\Gamma$ be a reducible simplicial complex on $[m]$ with decomposition $(\Gamma_1,S,\Gamma_2)$ and $\ddb = (d_1, \ldots, d_m) \in \NN^m$.  Suppose $\Gamma_1$ has the vertex set $\alpha=\{\alpha_1,\ldots,\alpha_k\}$ and $\Gamma_2$ has the vertex set $\beta=\{\beta_1,\ldots, \beta_s\}$. Then $S=\{\alpha_1,\ldots,\alpha_k\}\cap\{\beta_1,\ldots,\beta_s\}$. 
We also let $\ddb_\alpha = (d_{\alpha_1}, \ldots, d_{\alpha_k})$, with analogous definitions for $\ddb_\beta$ 
and $\ddb_S$.
Let $p$ be a point in $\Delta_{\Gamma, \ddb}$ and consider the maps
$$\pi_1:\Delta_{\Gamma, \ddb} \to \Delta_{\Gamma_1, \ddb_\alpha} 
\,\,\,\, p\mapsto p_1=p_{\{\alpha_1,\dots,\alpha_k\}}.$$
$$\pi_2:\Delta_{\Gamma, \ddb} \to \Delta_{\Gamma_2, 
\ddb_\beta} \,\,\,\, p\mapsto p_2=p_{\{\beta_1,\dots,\beta_s\}}.$$
More precisely, 
$$(\pi_1(p))_{i_\alpha} = \sum_{j \in \R: j_\alpha = i_\alpha} p_j \quad \mbox{  and } \quad
(\pi_2(p))_{i_\beta} = \sum_{j \in \R: j_\beta = i_\beta} p_j.$$

\begin{lemma} \label{lem:diagram-commutes}
Let $\Gamma$ be a reducible simplicial complex on $[m]$ with decomposition $(\Gamma_1,S,\Gamma_2)$ and $\ddb = (d_1, \ldots, d_m) \in \NN^m$. Let $q \in \M_{\Gamma, \ddb}$ so that $q_1 = \pi_1(q)$ and $q_2 = \pi_2(q)$. 
Furthermore, consider the maps
$$\pi_1':\Delta_{\Gamma_1, \ddb_\alpha} \to \Delta_{2^S, \ddb_S} 
\,\,\,\, p\mapsto p_{S}$$
$$\pi_2':\Delta_{\Gamma_1, \ddb_\beta} \to \Delta_{2^S, 
\ddb_S} \,\,\,\, p\mapsto p_{S}$$
defined by  
$$(\pi_1'(p))_{i_S} = \sum_{j \in \R_\alpha: j_S = i_S} p_j \quad \mbox{  and } \quad
(\pi_2'(p))_{i_S} = \sum_{j \in \R_\beta: j_S = i_S} p_j.$$ 
Then $q_1 \in \M_{\Gamma_1, \ddb_\alpha}$ and 
$q_2 \in \M_{\Gamma_2, \ddb_\beta}$, and 
the following diagram commutes:

\begin{center}
\begin{tikzcd}
                                                             & q\in\M_{\Gamma, \ddb} \arrow[ld, "\pi_1"', shift right] \arrow[rd, "\pi_2"] &                                                              \\
{q_1} \in\M_{\Gamma_1, \ddb_\alpha}\arrow[rd,"\pi_1'"'] &                                                                             & {q_2} \arrow[ld, "\pi_2'"]\in\M_{\Gamma_2, \ddb_\beta} \\
                                                             & q_3\in\M_{2^S, \ddb_S}                                                                         &                                                              \\
                                                             &                                      
\end{tikzcd}
\end{center}
\end{lemma}
\begin{proof}
By the definitions of the maps, $ \pi_1' \circ \pi_1 = \pi_2' \circ \pi_2$. 
Also, since $\M_{2^S, \ddb_S} = \Delta_{2^S, \ddb_S}$ it is clear that
$q_3 \in \M_{2^S,\ddb_S}$. We just need to show $q_1 \in \M_{\Gamma_1, \ddb_\alpha}$ and  $q_2 \in \M_{\Gamma_2, \ddb_\beta}$. We prove the first claim since the second one requires the same argument. 
Let $t \in\M_{\Gamma_1, \ddb_\alpha}$ be the MLE of $q_1$ and $r\in\M_{\Gamma_2, \ddb_\beta}$ be the MLE of $q_2$. We will show that $q_1=t$. Note that $q\in\M_{\Gamma, \ddb}$, so it is its own MLE in the model. Since $t$ is in the same logarithmic Voronoi polytope as $q_1$ and $r$ is in the same logarithmic
Voronoi polytope as $q_2$, we see that $\pi_1'(t)= t_S = q_3$ and $\pi_2'(r) = r_S = q_3$. 
Then by \cite[Prop 4.1.4]{Lau96}  
$$q_{i_1,\ldots, i_m}=\frac{(t_{i_\alpha})\cdot (r_{i_\beta})}{(r_S)_{i_S}},$$
where $i_{\alpha}=(i_{\alpha_1},\ldots, i_{\alpha_k})$ and $i_{\beta}=(i_{\beta_1},\ldots, i_{\beta_s})$.
Then observe that for any $i_\alpha$, we get
\begin{align*}
    (q_1)_{i_{\alpha}} &=\sum_{j\in \R: j_\alpha = i_\alpha}
    \frac{(t_{j_\alpha}) \cdot (r_{j_\beta})}{(r_S)_{j_S}}
    =\frac{t_{i_\alpha}}{(r_S)_{i_S}} \sum_{j_\beta:j_S=i_S} r_{j_\beta} = \frac{t_{i_\alpha} \cdot (r_S)_{i_S}}{(r_S)_{i_S}}
    =t_{i_{\alpha}}.
\end{align*}
Since $i_\alpha$ was arbitrary, we get that $q_1=t$. 
\end{proof}

%\begin{example}
%Consider the complex $\Gamma=[1][2]$ for $m=2$. Suppose %both $X_1$ and $X_2$ are discrete random variables. Let %$\Gamma_1=[1]$ and $\Gamma_2=[2]$. Given a table %$(u_{ij})\in\M_{\Gamma}$, we have that %$\varphi_1(u_{ij})=(u_{i+})$ and $\varphi_2(u_{ij})=%(u_{+j})$. Note that because $\Gamma_1=\Gamma_2$ is %simply a point, the models $\M_{\Gamma_1}$ and %$\M_{\Gamma_2}$ are entire simplices. Hence, for any %point $(p_{ij})\in\M_\Gamma$, the respective logarithmic %Voronoi cells of the points $(p_{i+})$ and $(p_{+j})$ in %$\M_{\Gamma_1}$ and $\M_{\Gamma_2}$ consist of only %those points themselves. Define
%$$\psi:\Delta_{\Gamma}\to\Delta_{\Gamma_1}\times\Delta_{\%Gamma_2}:(p_{ij})\to (\varphi_1(p),\varphi_2(p)).$$
%We claim that $Q_p^{\Gamma}=\psi^{-1}(p_1,p_2)$. Note %that by definition, 
%$$Q_p^{\Gamma}=\{(u_{ij})\in\Delta_{\Gamma}: %u_{i+}=p_{i+}\text{ and }u_{+j}=p_{+j}\text{ for all %}i,j\}.$$
%This is precisely the pre-image of $(p_1,p_2)$ under %$\psi$.
%\end{example}

\comm{
\begin{example}
Consider the complex $\Gamma=[12][13][23][24][34]$ for $m=4$. Suppose both $X_1, X_2$, $X_3$, and $X_4$ are discrete random variables. Let $\Gamma_1=[12][13][23]$ and $\Gamma_2=[23][24][34]$, so $S=\{2,3\}$. Define $\psi:\RR^{\Gamma}\to\RR^{\Gamma_1}\times\RR^{\Gamma_2}$ as follows:
$$\psi: (u_{ijk\ell})\mapsto ((u_{ijk+}), (u_{+jk\ell})).$$
\begin{thm} For any $p=(p_{ijk\ell})\in\M_{\Gamma}$ where $\Gamma$ is as in the example above,
\begin{align*}
\log\Vor_\Gamma(p)=&\Bigg[\left\{(u_{ijk\ell})\in\Delta_{\Gamma}: u_{ijk\ell}=\frac{v_{ijk}\cdot w_{jk\ell}}{p_{+jk+}}\text{ for }v\in\log\Vor_{\Gamma_1}(p_{ijk+})\text{ and }w\in\log\Vor_{\Gamma_2}(p_{ijk+})\right\}\\
&+\ker(\psi)\Bigg]\cap \Delta_{\Gamma}
\end{align*}
\end{thm}
\begin{proof}
We prove this by double containment. First, we show that the right-hand side is contained in the log-Voronoi cell at $p$. Let $u=u^{(1)}+u^{(2)}\in\Delta_{\Gamma}$, where $u^{(1)}_{ijk\ell}=\frac{v_{ijk}\cdot w_{jk\ell}}{p_{+jk+}}$ with $v\in\log\Vor_{\Gamma_1}(p_{ijk+})$, $w\in\log\Vor_{\Gamma_2}(p_{+jk\ell})$ for all $i,j,k,\ell$, and $u^{(2)}\in\ker(\psi)$. We need to show that $u|_{F}=p|_{F}$ for any facet $F\in\Gamma$. For simplicity, we let $F=\{1,2\}$. Then note that for all $i,j$, we have:
$$u^{(1)}_{ij++}=\sum_{k,\ell}\frac{v_{ijk}\cdot w_{jk\ell}}{p_{+jk+}}=\frac{v_{ij+}\cdot w_{j++}}{p_{+j++}}=v_{ij+}=p_{ij++}.$$
The argument is similar for all other facets, so we conclude that $u^{(1)}\in\log\Vor_{\Gamma}(p)$. Furthermore, note that $u_{ij++}=u^{(1)}_{ij++}+u^{(2)}_{ij++}=u^{(1)}_{ij++}$, since $u^{(2)}\in \ker(\psi)$ and so has zero marginals on every facet. Therefore, $u\in\log\Vor_{\Gamma}(p)$, as desired.

For the other direction, let $u\in\log\Vor_{\Gamma}(p)$. We will show that $u_{ijk\ell}=\frac{u_{ijk+}u_{+jk\ell}}{p_{+jk+}}+y_{ijk\ell}$ where $y\in\ker(\psi)$. (One can easily check that $u_{ijk+}\in\log\Vor_{\Gamma_1}(p_{ijk+})$ and $u_{+jk\ell}\in\log\Vor_{\Gamma_2}(p_{+jk\ell})$). Note that since $p$ is its own MLE in $\M_{\Gamma}$, we have that for any $i,j,k,\ell$,
$$p_{ijk\ell}=\frac{p_{ijk+}\cdot p_{+jk\ell}}{p_{+jk+}}.$$
Therefore, we also get
\begin{align*}
u_{ij++}=p_{ij++}&=\sum_{k,\ell}p_{ijk\ell}\\
&=\sum_{k,\ell}\frac{p_{ijk+}\cdot p_{+jk\ell}}{p_{+jk+}}\\
&=\sum_{k}\frac{p_{ijk+}\cdot p_{+jk+}}{p_{+jk+}}=p_{+jk+}\\
&=\sum_{k}\frac{u_{ijk+}\cdot u_{+jk+}}{p_{+jk+}}\\
&=\sum_{k,\ell}\frac{u_{ijk+}\cdot u_{+jk\ell}}{p_{+jk+}}.
\end{align*}
Similar equalities hold for all other facets of $\Gamma$. Hence, for any $i',j'$, we get that
$$u_{i'j'++}=\sum_{k,\ell}\frac{u_{i'j'k+}\cdot u_{+j'k\ell}}{p_{+j'k+}}=\sum_{k,\ell}u_{i'j'k\ell}.$$
This means that for any $k',\ell'$, the following are true:
$$\frac{u_{i'j'k'+}\cdot u_{+j'k'\ell'}}{p_{+j'k'+}}=u_{i'j'++}-\sum_{\substack{k\in[d_{3}]\setminus\{k'\}\\\ell\in[d_{4}]\setminus\{\ell'\}}}\frac{u_{i'j'k+}\cdot u_{+j'k\ell}}{p_{+j'k+}}$$
and
\begin{align*}
u_{i'j'k'\ell'}=u_{i'j'++}-\sum_{\substack{k\in[d_{3}]\setminus\{k'\}\\\ell\in[d_{4}]\setminus\{\ell'\}}}u_{i'j'k\ell}.
\end{align*}
Subtracting the above two equations and re-arranging, we get that
\begin{align}\label{sum-over-kl}
u_{i'j'k'\ell'}=\sum_{\substack{k\in[d_{3}]\setminus\{k'\}\\\ell\in[d_{4}]\setminus\{\ell'\}}}\Bigg(\frac{u_{i'j'k+}\cdot u_{+j'k\ell}}{p_{+j'k+}}-u_{i'j'k\ell}\Bigg)+\frac{u_{i'j'k'+}\cdot u_{+j'k'\ell'}}{p_{+j'k'+}}.
\end{align}

If we can show that the sum above is in the kernel of $\psi$, we're done. To do this, we need to show that both $[123]$- and $[234]$-marginals are zero. Observe that
\begin{align*}
\sum_{i'}\sum_{\substack{k\in[d_{3}]\setminus\{k'\}\\\ell\in[d_{4}]\setminus\{\ell'\}}}\Bigg(\frac{u_{i'j'k+}\cdot u_{+j'k\ell}}{p_{+j'k+}}-u_{i'j'k\ell}\Bigg)&=\sum_{\substack{k\in[d_{3}]\setminus\{k'\}\\\ell\in[d_{4}]\setminus\{\ell'\}}}\sum_{i'}\Bigg(\frac{u_{i'j'k+}\cdot u_{+j'k\ell}}{p_{+j'k+}}-u_{i'j'k\ell}\Bigg)\\
&=\sum_{\substack{k\in[d_{3}]\setminus\{k'\}\\\ell\in[d_{4}]\setminus\{\ell'\}}}\Bigg(\frac{u_{+j'k+}\cdot u_{+j'k\ell}}{p_{+j'k+}}-u_{+j'k\ell}\Bigg)\\
&=\sum_{\substack{k\in[d_{3}]\setminus\{k'\}\\\ell\in[d_{4}]\setminus\{\ell'\}}}(u_{+j'k\ell}-u_{+j'k\ell})=0,
\end{align*}
so the $[234]$-marginal is indeed zero. To show that the $[234]$-marginal is zero, note that by re-writing $u_{++k'\ell'}$, we have the equality
\begin{align}\label{sum-over-ij}
u_{i'j'k'\ell'}=\sum_{\substack{i\in[d_{1}]\setminus\{i'\}\\j\in[d_{2}]\setminus\{j'\}}}\Bigg(\frac{u_{ijk'+}\cdot u_{+jk'\ell'}}{p_{+jk'+}}-u_{ijk'\ell'}\Bigg)+\frac{u_{i'j'k'+}\cdot u_{+j'k'\ell'}}{p_{+j'k'+}}.
\end{align}
Hence, the two sums in (\ref{sum-over-kl}) and (\ref{sum-over-ij}) are equal. We also observe that
\begin{align*}
\sum_{\ell'}\sum_{\substack{i\in[d_{1}]\setminus\{i'\}\\j\in[d_{2}]\setminus\{j'\}}}\Bigg(\frac{u_{ijk'+}\cdot u_{+jk'\ell'}}{p_{+jk'+}}-u_{ijk'\ell'}\Bigg)&=\sum_{\substack{i\in[d_{1}]\setminus\{i'\}\\j\in[d_{2}]\setminus\{j'\}}}\sum_{\ell'}\Bigg(\frac{u_{ijk'+}\cdot u_{+jk'
\ell'}}{p_{+jk'+}}-u_{ijk'\ell'}\Bigg)\\
&=\sum_{\substack{i\in[d_{1}]\setminus\{i'\}\\j\in[d_{2}]\setminus\{j'\}}}\Bigg(\frac{u_{ijk'+}\cdot u_{+jk'+}}{p_{+jk'+}}-u_{ijk'+}\Bigg)\\
&=\sum_{\substack{i\in[d_{1}]\setminus\{i'\}\\j\in[d_{2}]\setminus\{j'\}}}(u_{ijk'+}-u_{ijk'+})=0,
\end{align*}
so the $[123]$-marginal is also zero.
Hence, we have that $u_{i'j'k'\ell'}=\frac{u_{i'j'k'+}\cdot u_{+j'k'\ell'}}{p_{+j'k'+}}+y_{i'j'k'\ell'}$ where $$y:=\sum_{\substack{k\in[d_{3}]\setminus\{k'\}\\\ell\in[d_{4}]\setminus\{\ell'\}}}\Bigg(\frac{u_{i'j'k+}\cdot u_{+j'k\ell}}{p_{+j'k+}}-u_{i'j'k\ell}\Bigg)\in\ker(\psi),$$ as desired.
\end{proof}
\end{example}
}

Now we are ready to prove the main result of this section. From the discussion so far we see that $\pi_1$ and $\pi_2$ restrict to logarithmic Voronoi polytopes, i.e., 
$\pi_1 \, : \, Q_p^\Gamma \longrightarrow Q_{p_1}^{\Gamma_1}$
and $\pi_2 \, : \, Q_p^\Gamma \longrightarrow Q_{p_2}^{\Gamma_2}$ where $p_1 = \pi_1(p)$ and $p_2 = \pi_2(p) $ for $p \in \Delta_{\Gamma, \ddb}$. In fact, we can
take $q = p \in \M_{\Gamma, \ddb}$ so that $q_1$ and $q_2$ 
are in $\M_{\Gamma_1, \ddb_\alpha}$ and $\M_{\Gamma_2, \ddb_\beta}$, respectively, by the above lemma. The next theorem reconstructs $Q_p^{\Gamma}$ from the 
logarithmic Voronoi polytopes $Q_{p_1}^{\Gamma_1}$ 
and $Q_{p_2}^{\Gamma_2}$.

\begin{thm}\label{thm:decomposition-theorem} 
Let $\Gamma$ be a reducible simplicial complex on $[m]$ with decomposition $(\Gamma_1,S,\Gamma_2)$ and $\ddb = (d_1, \ldots, d_m) \in \NN^m$.
 Let $\psi:\Delta_{\Gamma, \ddb}\to\Delta_{\Gamma_1, \ddb_\alpha}\times\Delta_{\Gamma_2, \ddb_\beta}$ be the map
$\psi(u) = (\pi_1(u), \pi_2(u))$.
Then for any $q\in\M_{\Gamma, \ddb}$, we have
\begin{align*}
Q_q^{\Gamma}=&\Bigg[\left\{u\in\Delta_{\Gamma, \ddb}: u_{i_1,\cdots, i_m}=\frac{v_{i_{\alpha}}\cdot w_{i_{\beta}}}{(q_S)_{i_S}}\text{ for }v\in Q_{q_1}^{\Gamma_1}\text{ and }w\in Q_{q_2}^{\Gamma_2}\right\}+\ker(\psi)\Bigg]\cap \Delta_{\Gamma, \ddb}
\end{align*}
\end{thm}
\begin{proof}
We proceed by double containment. To show that the right-hand side is contained in $Q_{q}^{\Gamma}$, let $u=u^{(1)}+u^{(2)}\in\Delta_{\Gamma, \ddb}$ where $u^{(1)}_{i_1,\cdots, i_m}=\frac{v_{i_\alpha}\cdot w_{i_\beta}}{(q_S)_{i_S}}$ for $v\in Q_{q_1}^{\Gamma_1}$, $w\in Q_{q_2}^{\Gamma_2}$ and $u^{(2)}\in\ker(\psi)$. Let $F = \{f_1, \ldots, f_k\}$ be any facet of $\Gamma$. Then $F$ is either in $\Gamma_1$ or in $\Gamma_2$. Without loss of generality, assume $F$ is in $\Gamma_1$. Then for any $i_F = (i_{f_1}, \ldots, i_{f_k})$, we have
\begin{align*}
    ((u^{(1)})_F)_{i_F}&=\sum_{\{j_\alpha:j_F=i_F\}} \left( \sum_{\{j \in \R :j_\alpha = i_\alpha, j_F=i_F\}}\frac{v_{j_\alpha}\cdot w_{j_\beta}}{(q_S)_{j_S}} \right) =\sum_{\{j_\alpha:j_F=i_F\}} \frac{v_{j_\alpha}\cdot (w_S)_{j_S}}{(q_S)_{j_S}}\\&=\sum_{\{j_\alpha:j_F=i_F\}} v_{j_\alpha}= (v_F)_{i_F}=(q_F)_{i_F}.
\end{align*}

Hence $u^{(1)}\in Q_q^{\Gamma}$. But since $u^{(2)}\in\ker(\psi)$, it has a zero $F$-marginal for every facet of $\Gamma$. Thus, $u=u^{(1)}+u^{(2)}\in Q_q^{\Gamma}$, as desired.

To show the reverse containment, let $u\in Q_q^{\Gamma}$ and let $v\in\Delta_{\Gamma, \ddb}$ be the point defined by
$$v_{i_1,\cdots, i_m} =\frac{(u_\alpha)_{i_\alpha}\cdot (u_\beta)_{i_\beta}}{(q_S)_{i_S}}$$
for all $(i_1,\ldots,i_m)$. We write $u=v+(u-v)$.
Since $u_\alpha = \pi_1(u) \in Q_{q_1}^{\Gamma_1}$
and $u_\beta = \pi_2(u) \in Q_{q_2}^{\Gamma_2}$, it  suffices to show that $u-v\in \ker\psi$. That is, we must show that $[(u-v)_\alpha]_{i_\alpha}=0$ and $[(u-v)_\beta]_{i_\beta}=0$. For any $i_{\alpha}$, $[(u-v)_\alpha]_{i_\alpha}$ is equal to
$$(u_\alpha)_{i_\alpha}-\sum_{\{j \in \R: j_\alpha = i_\alpha\}}\frac{ (u_\alpha)_{j_\alpha}\cdot (u_\beta)_{j_\beta}}{(q_S)_{j_S}}= (u_\alpha)_{i_\alpha}-\frac{(u_\alpha)_{i_\alpha}}{(q_S)_{i_S}} \sum_{\{j_\beta: j_S = i_S\}} (u_\beta)_{j_\beta} = (u_\alpha)_{i_\alpha}- \frac{(u_\alpha)_{i_\alpha} \cdot (u_S)_{i_S}}{(q_S)_{i_S}} =0.$$
Similarly, one shows that $(u-v)_{i_\beta}=0$ as well. Thus, $u-v\in\ker\psi$, and this concludes the proof.
\end{proof}
In this theorem, the first summand in the Minkowski sum
that appears in the decomposition of $Q_q^\Gamma$ is an 
interesting object. It is nonlinear and captures  a portion 
of $Q_q^\Gamma$. 
\begin{defn} \label{defn:product}
 Let $\Gamma$ be a reducible simplicial complex on $[m]$ with decomposition $(\Gamma_1,S,\Gamma_2)$ and $\ddb = (d_1, \ldots, d_m) \in \NN^m$. Let $p \in \Delta_{\Gamma, \ddb}$ and $p_i = \pi_i(p)$ for $i=1,2$. Then 
 the product of $Q_{p_1}^{\Gamma_1}$ and $Q_{p_2}^{\Gamma_2}$
 is defined as 
 \begin{align*}
Q_{p_1}^{\Gamma_1} \otimes_p Q_{p_2}^{\Gamma_2} =&\left\{u\in\Delta_{\Gamma, \ddb}: u_{i_1,\cdots, i_m}=\frac{v_{i_{\alpha}}\cdot w_{i_{\beta}}}{(p_S)_{i_S}}\text{ for }v\in Q_{p_1}^{\Gamma_1}\text{ and }w\in Q_{p_2}^{\Gamma_2}\right\}.
\end{align*}
\end{defn}
\begin{remark}
If $p' \in Q_p^\Gamma$ we get the equality 
of the logarithmic Voronoi polytopes $Q_{p'}^\Gamma =
Q_p^\Gamma$. Moreover, since $p_i' = \pi_i(p') \in Q_{p_i}^{\Gamma_i}$ for $i=1,2$, we see that
$Q_{p_i'}^{\Gamma_i} = Q_{p_i}^{\Gamma_i}$. Therefore,
$Q_{p_1}^{\Gamma_1} \otimes_p Q_{p_2}^{\Gamma_2} = 
Q_{p_1'}^{\Gamma_1} \otimes_{p'} Q_{p_2'}^{\Gamma_2}$.
In other words, the product depends only on the logarithmic Voronoi polytope $Q_p^\Gamma$ and not on the individual points in the polytope. 
\end{remark}

\begin{example}
Consider the complex $\Gamma=[12][13][23][24][34]$ for $m=4$. Suppose both $X_1, X_2$, $X_3$, and $X_4$ are binary random variables, i.e., $d_i = 2$ for all $i=1, \ldots, 4$. Let $\Gamma_1=[12][13][23]$ and $\Gamma_2=[23][24][34]$, so $S=\{2,3\}$. The logarithmic Voronoi polytopes $Q_{p_1}^{\Gamma_1}$
and $Q_{p_2}^{\Gamma_2}$ have dimension one whereas 
the dimension of $Q_p^{\Gamma}$ is six. This is consistent with
Theorem \ref{thm:decomposition-theorem} since $Q_{p_1}^{\Gamma_1} \otimes_p Q_{p_2}^{\Gamma_2}$ is a two-dimensional surface in
$\Delta_{15}$ and $\ker(\psi)$ has dimension four. More explicitly, if $v = (v_{ijk})$ and $w = (w_{jk\ell})$ are 
points in $Q_{p_1}^{\Gamma_1}$ and $Q_{p_2}^{\Gamma_2}$, respectively, where in particular $v_{+jk} = v_{1jk} + v_{2jk}$
and $w_{jk+} = w_{jk1} + w_{jk2}$ are equal to each other for
all $j,k = 1,2$, then $Q_{p_1}^{\Gamma_1} \otimes_p Q_{p_2}^{\Gamma_2}$ consists of points $u = (u_{ijk\ell})$
where 
$$u_{ijk \ell} = \frac{v_{ijk} \cdot w_{jk\ell}}{v_{+jk}}. $$
\end{example}

\subsection{Comparing divergences}\label{sec:comparing-divergences}
Since a reducible model $\M_{\Gamma, \ddb}$ associated
to a simplicial complex $\Gamma$ on $[m]$ with decomposition
$(\Gamma_1, S, \Gamma_2)$ has the two associated models
$\M_{\Gamma_1, \ddb_\alpha}$ and $\M_{\Gamma_2, \ddb_\beta}$
it is natural to ask how the divergences from these three models are related. Before we present our contributions we wish to cite two results of Mat\'{u}\v{s} that are relevant.

\begin{prop} \label{prop:divergence-bound-1} \cite[Lemma 3]{Mat09} For any $p 
\in \Delta_{\Gamma, \ddb}$ and a reducible model 
$\M_{\Gamma, \ddb}$, 
$$ D_{\M_{\Gamma, \ddb}}(p) = D_{\M_{\Gamma_1, \ddb_\alpha}}(\pi_1(p)) + D_{\M_{\Gamma_2, \ddb_\beta}}(\pi_2(p)) + H(\pi_1(p)) + H(\pi_2(p)) - H(p) - H((\pi_1' \circ \pi_1)(p))$$
where $H(\cdot)$ is the entropy.     
\end{prop}

\begin{prop} \label{prop:divergence-bound-2} \cite[Corollary 3]{Mat09} For a hierarchical 
log-linear 
model $\M_{\Gamma, \ddb}$ we have 
$$ D(\M_{\Gamma, \ddb}) \leq \underset{F \mbox{ facet of } \Gamma}{\min}  \left\{ \sum_{i \not \in F} \log d_i \right\}. $$
\end{prop}

With regards to Proposition \ref{prop:divergence-bound-1}
we point out that the four entropy terms together give a nonnegative quantity because of the strong subadditivity property of entropy. Therefore, for a reducible model 
we get the inequality $ D_{\M_{\Gamma, \ddb}}(p) \geq D_{\M_{\Gamma_1, \ddb_\alpha}}(\pi_1(p)) + D_{\M_{\Gamma_2, \ddb_\beta}}(\pi_2(p))$. We state and prove a similar
inequality in Corollary \ref{cor:bound-reducible}. In the case when the point $p$ lives in the product portion of its logarithmic Voronoi polytope (as in Theorem \ref{thm:decomposition-theorem}), we recover the equality below.
\begin{prop} \label{prop:divergence-reducible}
 Let $\Gamma$ be a reducible simplicial complex on $[m]$ with decomposition $(\Gamma_1,S,\Gamma_2)$ and $\ddb = (d_1, \ldots, d_m) \in \NN^m$. Let $p \in \Delta_{\Gamma, \ddb}$ and $p_i = \pi_i(p)$ for $i=1,2$.  If
  $u = v \otimes_p w$ where $v \in Q_{p_1}^{\Gamma_1}$
 and $w \in Q_{p_2}^{\Gamma_2}$, then
 $D_{\M_{\Gamma, \ddb}}(u) = D_{\M_{\Gamma_1, \ddb_\alpha}}(v) +  D_{\M_{\Gamma_2, \ddb_\beta}}(w)$.
\end{prop}
\begin{proof}
Let $t \in \M_{\Gamma_1, \ddb_\alpha}$ and $r \in \M_{\Gamma_2, \ddb_\beta}$ be the respective 
maximum likelihood estimators of $v$ and $w$. Similarly, 
let $q \in \M_{\Gamma, \ddb}$ be the maximum likelihood
estimator of $u$. By \cite[Prop 4.1.4]{Lau96}, $q = t \otimes_p r$ and 
$$D_{\M_{\Gamma, \ddb}}(u) \, = \, \sum_{i \in \R} u_i \log(u_i / q_i) \, = \, \sum_{i \in \R} u_i \log(v_{i_\alpha} / t_{i_\alpha}) + \sum_{i \in \R} u_i \log(w_{i_\beta} / r_{i_\beta}).$$
Since $(u_\alpha)_{i_\alpha} = \sum_{j \in R: j_\alpha = i_\alpha} u_j = v_{i_\alpha}$ and $(u_\beta)_{i_\beta} = \sum_{j \in R: j_\beta = i_\beta} u_j = w_{i_\beta}$ (see
the proof of Lemma \ref{lem:diagram-commutes})
we conclude that
$$D_{\M_{\Gamma, \ddb}}(u) \, = \, \sum_{i_\alpha} v_{i_\alpha} \log(v_{i_\alpha} / t_{i_\alpha}) +  \sum_{i_\beta} w_{i_\beta} \log(w_{i_\beta} / r_{i_\beta}) = D_{\M_{\Gamma_1, \ddb_\alpha}}(v) +  D_{\M_{\Gamma_2, \ddb_\beta}}(w).$$
\end{proof}
\begin{cor}
Let $\Gamma$ be a reducible simplicial complex on $[m]$ with decomposition $(\Gamma_1,S,\Gamma_2)$ and $\ddb = (d_1, \ldots, d_m) \in \NN^m$. Let $p \in \Delta_{\Gamma, \ddb}$ and $p_i = \pi_i(p)$ for $i=1,2$. Suppose $v \in Q_{p_1}^{\Gamma_1}$ maximizes the divergence 
to $\M_{\Gamma_1, \ddb_\alpha}$ over all points in $Q_{p_1}^{\Gamma_1}$. Similarly, suppose $w \in Q_{p_2}^{\Gamma_2}$ and $p' \in Q_p^\Gamma$ be such maximizers. Then $D_{\M_{\Gamma, \ddb}}(p') \geq D_{\M_{\Gamma_1, \ddb_\alpha}}(v) +  D_{\M_{\Gamma_2, \ddb_\beta}}(w)$.
\end{cor}    
\begin{proof}
 Use Proposition \ref{prop:divergence-reducible} with
 $D_{\M_{\Gamma, \ddb}}(p') \geq D_{\M_{\Gamma, \ddb}}(u)$
 where $u = v \otimes_p w$. 
\end{proof}
The corollary has the following implication for the maximum
divergence to a reducible model. 
\begin{cor} \label{cor:bound-reducible}
Let $\Gamma$ be a reducible simplicial complex on $[m]$ with decomposition $(\Gamma_1,S,\Gamma_2)$ and $\ddb = (d_1, \ldots, d_m) \in \NN^m$. Let $p \in \Delta_{\Gamma, \ddb}$ be a point 
that attains the maximum divergence $D(\M_{\Gamma, \ddb})$ and $p_i = \pi_i(p)$ for $i=1,2$. If $Q_{p_1}^{\Gamma_1}$ and
$Q_{p_2}^{\Gamma_2}$ contain points which attain the 
maximum divergence $D(\M_{\Gamma_1, \ddb_\alpha})$ and
$D(\M_{\Gamma_2, \ddb_\beta})$, respectively, then 
$D(\M_{\Gamma, \ddb}) \geq D(\M_{\Gamma_1, \ddb_\alpha}) +  D(\M_{\Gamma_2, \ddb_\beta})$.
\end{cor}

\subsection{Independence and related models}
We have already encountered  an independence model in Example \ref{ex:2x3-chamber}. More generally, for discrete
random variables $X_1, \ldots, X_m$ with respective state 
spaces $[d_i]$, the independence model is the hierarchical
log-linear model on $[m]$ associated to the simplicial complex $\Gamma$ consisting of just the $m$ vertices $\Gamma = [1][2]\cdots [m]$. This is a reducible (in fact decomposable) model. Proposition \ref{prop:divergence-bound-2}
immediately implies the following (see also \cite{AK06}).
\begin{cor} \label{cor:independence-bound}
Let $\M$ be the independence model of $m$ discrete random variables $X_1,\ldots, X_m$ with state spaces $[d_i]$, respectively, where $d_1 \leq d_2 \leq \cdots \leq d_m$. Then
$$ D(\M) \leq \log d_1 + \cdots + \log d_{m-1}.$$
\end{cor}
Those independence models which achieve the upper bound in
this result have been characterized \cite[Theorem 4.4]{AK06}. 
For instance, when $m=2$ as well as in the case of
$d_1 = \cdots = d_m$, the upper bound is achieved. Since we will use it later we record a precise result regarding the latter case. For this, let $S_{d+1}$ denote the group of permutations on $\{0,1,\ldots, d\}$ and let $\delta$ denote the Dirac delta (indicator) function.
\begin{theorem}\label{thm:multi-information-max}\cite{AK06}
Let $\M$ be the independence model of $m$ $(d+1)$-ary random variables. Then the maximum divergence from $\M$ is $D(\M)=(m-1)\log (d+1)$. This maximum value is achieved at vertices of the logarithmic Voronoi polytope at the unique point $q=\left(\frac{1}{(d+1)^m},\ldots,\frac{1}{(d+1)^m}\right)$. Each maximizer has the form $\frac{1}{d+1}\sum_{j=0}^{d}\delta_{j,\sigma_2(j),\ldots, \sigma_{m}(j)}$ where $\sigma_i\in S_{d+1}$ for all $i=2,\ldots, m$. 
\end{theorem}
Now we wish to illustrate the utility of Corollary \ref{cor:bound-reducible} for independence models in two examples. 
\begin{example} \label{ex:independence2x2x2}
The independence model $\M_{222}$ of three binary variables is a $3$-dimensional model in $\Delta_7$. We denote the coordinates
of the points in $\Delta_7$ by $p_{ijk}$ where $i,j,k = 1,2$. 
By using Corollary \ref{cor:bound-reducible} we can
show that $D(\M_{222}) = 2 \log 2$. Note that this model is
reducible with $\Gamma = (\Gamma_1, S, \Gamma_2)$ where
$\Gamma_1 = [1][2]$, $\Gamma_2 = [2][3]$, and $S= \{2\}$. 
The models $\M_{\Gamma_1}$ and $\M_{\Gamma_2}$ are themselves
independence models of two binary variables. In Example \ref{ex:two-binary-RV} we saw that  $D(\M_{\Gamma_1}) = D(\M_{\Gamma_2}) = \log 2$ where there are exactly two maximizers 
$$v= (v_{11}, v_{12}, v_{21}, v_{22}) = (\frac{1}{2}, 0, 0, \frac{1}{2}) \mbox{  and  } 
w = (w_{11}, w_{12}, w_{21}, w_{22}) = (\frac{1}{2}, 0, 0, \frac{1}{2}).$$
We will view $v$ and $w$ as elements of the logarithmic Voronoi
polytopes $Q_v^{\Gamma_1}$ and $Q_w^{\Gamma_2}$, respectively. 
These two polytopes are {\it compatible} in the sense that 
$\pi_1'(v) = (v_{+1}, v_{+2}) = (\frac{1}{2}, \frac{1}{2})$
is equal to $\pi_2'(w) = (w_{1+}, w_{2+}) = (\frac{1}{2}, \frac{1}{2})$. In other words, there exists a logarithmic Voronoi polytope 
$Q_p^{\Gamma}$ such that $Q_v^{\Gamma_1} \otimes Q_w^{\Gamma_2} \subset Q_p^{\Gamma}$. Here $p = v \otimes w$ where 
$p_{ijk} = \frac{v_{ij} w_{jk}}{v_{+j}}$, and we see that
$p_{112} = p_{221} = \frac{1}{2}$. Since 
$D_{\M_{222}}(p) = D(\M_{\Gamma_1}) + D(\M_{\Gamma_2}) = \log 2 + \log 2$ we conclude that $D(\M_{222}) = 2 \log 2$.
\end{example}
\begin{example} \label{ex:independence2x3x3}
Now we consider the independence model $M_{233}$. Corollary \ref{cor:independence-bound} states that $D(\M_{233}) \leq \log 2 + \log 3$, but this bound cannot be attained by 
\cite[Theorem 4.4]{AK06}. We wish to provide a rationale based on 
Corollary \ref{cor:bound-reducible}. The model $\M_{\Gamma_1}$
is the independence model of a binary and a ternary random
variables, and the model $\M_{\Gamma_2}$ is the independence 
model of two ternary random variables. By Example \ref{ex:2x3-chamber}, there are six types of divergence maximizers for $\M_{\Gamma_1}$. If we denote the points in $\Delta_{5}$ in which
$\M_{\Gamma_1}$ is contained by $v = (v_{11}, v_{12}, v_{13}; v_{21}, v_{22}, v_{23})$ these maximizers are
\begin{align*}
    (\frac{1}{2}, 0,0;0, \frac{r}{2}, \frac{1-r}{2}) & \quad  \quad
    (0, \frac{r}{2}, \frac{1-r}{2}; \frac{1}{2}, 0,0) \\
    (0,\frac{1}{2},0;\frac{r}{2}, 0, \frac{1-r}{2}) & \quad \quad
    (\frac{r}{2}, 0, \frac{1-r}{2}; 0, \frac{1}{2},0) \\
    (0,0, \frac{1}{2};\frac{r}{2},\frac{1-r}{2},0) &  \quad \quad
    (\frac{r}{2},\frac{1-r}{2},0;  0, 0, \frac{1}{2}) 
\end{align*}
where $0 < r < 1$. According to Theorem \ref{thm:multi-information-max}, there are six divergence maximizers 
of $\M_{\Gamma_2} \subset \Delta_8$. If we denote 
the points in $\Delta_8$ by $w = (w_{jk} \, : \, j,k = 1,2,3)$
these maximizers are $w^\sigma$ for each $\sigma \in S_3$ given
by 
\begin{align*}
     &  w_{11}^{id} = w_{22}^{id} = w_{33}^{id}  = \frac{1} {3}   
    & w_{12}^{(1 \, 2)} = w_{21}^{(1 \, 2)} = w_{33}^{(1 \, 2)}  = \frac{1}{3} \\
     &  w_{13}^{(1 \, 3)} = w_{22}^{(1 \, 3)} = w_{31}^{(1 \, 3)}  = \frac{1}{3}   & w_{11}^{(2 \, 3)} = w_{23}^{(2 \, 3)} = w_{32}^{(2 \, 3)} = \frac{1}{3} \\
    & w_{12}^{(1 \, 2 \, 3)} = w_{23}^{(1 \, 2 \, 3)} = w_{31}^{(1 \, 2 \, 3)}  = \frac{1}{3}   & w_{13}^{(1 \, 3 \, 2)} = w_{21}^{(1 \, 3 \, 2)} = w_{32}^{(1 \, 3 \, 2)} = \frac{1}{3}
\end{align*}   
Now we see that $\pi_1'(v) = (v_{+1}, v_{+2}, v_{+3})$
and $\pi_2'(w^\sigma) = (w_{1+}^\sigma, w_{2+}^\sigma, w_{3+}^\sigma) = (\frac{1}{3}, \frac{1}{3}, \frac{1}{3})$
are not equal to each other for any choice of the 
maximizer $v$ of $\M_{\Gamma_1}$ and $w^\sigma$ of $\M_{\Gamma_2}$. This means that $Q_v^{\Gamma_1}$ and 
$Q_{w^\sigma}^{\Gamma_2}$ are not compatible. In other words, it is impossible to apply Corollary \ref{cor:bound-reducible}.  Indeed, the bound cannot
be attained as it was explicitly shown in \cite[Example 20]{Rauh11}. The maximum divergence is equal to $\log(3+2\sqrt{2}) < \log 6 = \log 2 + \log 3$. Up to symmetry there is a unique global 
maximizer given by
$$p_{111} = \sqrt{2}-1, \,\,\, p_{222} = p_{233} = 1 - \frac{\sqrt{2}}{2}.$$
We believe that for reducible models induced by $\Gamma = (\Gamma_1, S, \Gamma_2)$ finding {\it compatible} logarithmic
Voronoi polytopes $Q_v^{\Gamma_1}$ and $Q_w^{\Gamma_2}$
that will give meaningful bounds for $D(\M_{\Gamma})$ is worth
exploring.
\end{example}

We close our discussion of reducible hierarchical log-linear models with a result involving conditional independence. 
For this we consider three random variables $X_1$, $X_2$, and $X_3$ with state spaces $[d_1], [d_2]$, and $[d_3]$, respectively. We let $\ddb = (d_1, d_2, d_3)$. The simplicial complex we will use is $\Gamma=[12][23]$
with the decomposition $([12], \{2\}, [23])$. The toric model 
$\M_{\Gamma, \ddb}$ consists of the joint probability distributions where $ X_1 \ci X_3 \, | \, X_2$. 

\begin{prop} \label{prop:conditional-indp}
  Let $\Gamma$ be the reducible simplicial complex on $[3]$ with decomposition $([12],\{2\}, [23]) $ and $\ddb = (d_1, d_2, d_3) \in \NN^3$. Then $D(\M_{\Gamma, \ddb}) = \min (\log d_1, \log d_3)$. 
\end{prop}
\begin{proof}
Proposition \ref{prop:divergence-bound-2} implies
that $D(\M_{\Gamma, \ddb}) \leq \min (\log d_1,  \log d_3)$.
The $0/1$ matrix $A_{\Gamma, \ddb}$ defining the model can be organized as follows. Recall that this model is defined
by the parametrization $p_{ijk} = a_{ij}b_{jk}$ with $i \in [d_1], j \in [d_2]$, and $k \in [d_3]$. We order the indices
$(i,j,k)$ lexicographically as follows: $(i,j,k) < (i', j', k')$ if $j<j'$, or if $j=j'$ and $i<i'$, or if $j=j'$ and $i=i'$ and $k<k'$. We sort the columns of $A_{\Gamma, \ddb}$
with respect to this ordering. We will also sort the rows 
of the matrix into $d_2$ blocks where in block $j$ we list
first the rows corresponding to the parameters $a_{ij}$ with $i=1, \ldots, d_1$ and 
then the parameters $b_{jk}$  with $k=1, \ldots, d_3$.
Then 
$$A_{\Gamma, \ddb} = \left( \begin{array}{cccc} A_{d_1,d_3} & 0 & \cdots & 0\\ 0 & A_{d_1,d_3} & \cdots & 0 \\
 \vdots & \vdots & \ddots & \vdots \\
 0 & 0 & \cdots & A_{d_1,d_3} \end{array} \right)$$
 where $A_{d_1,d_3}$ is the matrix defining an independence
 model of two random variables with state spaces $[d_1]$ and $[d_3]$. Since the maximum divergence from such a model is 
 $\min (\log d_1, \log d_3)$, and each block gives the same maximum divergence, Theorem \ref{thm:block-matrix} 
 implies the result.
\end{proof}

\section{Models of ML degree one}\label{sec:ml-deg-one} 
\begin{definition}\label{def:ml-degree}
The maximum likelihood degree (\textit{ML degree}) of a toric model $\M_A \subset \Delta_{n-1}$ is the number of points in $\CC^n$ that are in the intersection of the toric variety $X_A$ and the affine span of the logarithmic Voronoi polytope $Q_b$, where $b=Au$ for a generic $u\in\Delta_{n-1}$. 
\end{definition}
The definition above is equivalent to the standard definition in \cite[p. 140]{AlgStatBook}. Moreover, a model has ML degree one if and only if the maximum likelihood estimate of any data $u\in\Delta_{n-1}$ can be expressed as a rational function of the coordinates of $u$ \cite{HuhSturmfels2014LikelihoodGeometry}. We study the maximum divergence to such models in this section. Two dimensional 
models (toric surfaces) of ML degree one were classified in \cite{Davies_2022} along with some families of three dimensional models. We treat these families and the generalizations 
of some of them.

\subsection{Multinomial distributions}
We start by considering $m$ independent identically distributed $(d+1)$-ary random variables $X_1,\ldots, X_m$ with state spaces $\{0,\ldots,d\}$. Let $s_j$ be the probability of state $j$, and let $p_{i_0\ldots i_{d}}$ denote the probability of observing exactly $i_j$ occurrences of state $j$ for each $j\in\{0,1,\ldots,d\}$. Thus, $p_{i_0\ldots i_{d}}=\binom{m}{i_0,\ldots ,i_d}s_0^{i_0}\ldots s_d^{i_d}$ and we have the $d$-dimensional toric model parametrized as
$$\varphi:\Delta_{d}\to\Delta_{n-1}:(s_0,\ldots,s_d)\mapsto (p_{i_0 \ldots i_{d}}: \Sigma {i_j} = m)$$ where $n=\binom{m+d}{d}$. We refer to the the Zariski closure of the image of $\varphi$ as the \textit{twisted Veronese model} and denote it by $\V_{d,m}$. The columns of the matrix $A$ corresponding to the parametrization are the nonnegative integer solutions to $i_0+\cdots+i_d=m$. Note that $A$ has the constant vector $n\mathds{1}$ in its rowspan. Moreover, it has the same rowspan as the matrix $A'$ whose columns are of the form $(1,v)$ where $v\in\RR^{d}$ is a nonnegative integer solution to the inequality $i_1+\cdots+i_d\leq m$. Thus, geometrically, the model is given by all lattice points in the convex hull of $\{0,me_1,\ldots,me_d\}\subseteq \RR^{d}$, a $d$-dimensional simplex dilated by a factor of $m$. Hence $\V_{d,m}$ is isomorphic to the Veronese variety, except the weights are modified so that $\V_{d,m}$ has ML degree one. That is, if we were to change all multinomial coefficients in the definition of $p_{i_0 \ldots i_{d}}$ to 1, we would recover the usual Veronese variety. 

Since the ML degree of $\V_{d,m}$ is one, the MLE can be expressed as a rational function of the data. Fix $u\in \Delta_{n-1}$ and suppose $u$ is in the logarithmic Voronoi polytope $Q_p$ at some unknown $p=(p_1,\ldots, p_n)\in\V_{d,m}$. Let $Ap=b$ where $b=(b_0,\ldots,b_d)$ is the point corresponding to $p$ in  $\conv(A)$. Then each coordinate of the MLE of $u$ can be expressed as 
\begin{align}\label{eq:mle-twisted-veronese}
    q_{i_0,\ldots,i_d}=\binom{m}{i_0,\ldots ,i_d} \left(\frac{b_0}{m}\right)^{i_0}\cdots \left(\frac{b_d}{m}\right)^{i_d}.
\end{align}

\begin{example}[$d=2,m=3$]
Consider the twisted Veronese variety $\V_{2,3}$. As discussed above, the defining matrices $A$ and $A'$ can be written as either
$$A=\begin{pmatrix}
3&2&2&1&1&1&0&0&0&0\\
0&1&0&2&1&0&3&2&1&0\\
0&0&1&0&1&2&0&1&2&3
\end{pmatrix} \text{ or }
A'=\begin{pmatrix}
1&1&1&1&1&1&1&1&1&1\\
0&1&0&2&1&0&3&2&1&0\\
0&0&1&0&1&2&0&1&2&3
\end{pmatrix}.$$ In our computations, we will usually use the matrix $A$. The polytope associated to $\V_{2,3}$ is plotted in Figure \ref{fig:twisted-veronese} on the left. 

    \begin{figure}[H]
        \centering
\begin{tikzpicture}[scale=1]
\draw[gray, thick] (0,0) -- (0,3);
\draw[gray, thick] (0,0) -- (3,0);
\draw[gray, thick] (0,3) -- (3,0);

\draw[cyan, thin] (0,0) -- (3/2,3/2);
\draw[cyan, thin] (3,0) -- (0,3/2);
\draw[cyan, thin] (3/2,0) -- (0,3);

\filldraw[black] (0,0) circle (1pt);
\filldraw[black] (0,1) circle (1pt);
\filldraw[black] (0,2) circle (1pt);
\filldraw[black] (0,3) circle (1pt);
\filldraw[black] (1,0) circle (1pt);
\filldraw[black] (2,0) circle (1pt);
\filldraw[black] (3,0) circle (1pt);
\filldraw[red] (1,1) circle (1.5pt) node[anchor=south] {\scriptsize{$b$}};
\filldraw[black] (1,2) circle (1pt);
\filldraw[black] (2,1) circle (1pt);

\end{tikzpicture} \;\;\;\;\;\; \begin{tikzpicture}[scale=1.1]
\draw[gray, thick] (0,0,0) -- (0,0,2);
\draw[gray, thick] (0,0,0) -- (2,0,0);
\draw[gray, thick] (0,0,0) -- (0,2,0);
\draw[gray, thick] (2,0,0) -- (0,2,0);
\draw[gray, thick] (0,0,2) -- (0,2,0);
\draw[gray, thick] (0,0,2) -- (2,0,0);

\filldraw[black] (0,0,0) circle (1pt);
\filldraw[black] (1,0,0) circle (1pt);
\filldraw[black] (2,0,0) circle (1pt);
\filldraw[black] (0,1,0) circle (1pt);
\filldraw[black] (0,2,0) circle (1pt);
\filldraw[black] (0,0,1) circle (1pt);
\filldraw[black] (0,0,2) circle (1pt);
\filldraw[black] (1,1,0) circle (1pt);
\filldraw[black] (1,0,1) circle (1pt);
\filldraw[black] (0,1,1) circle (1pt);

\draw[cyan, thin] (0,0,0) -- (2/3,2/3,2/3);
\draw[cyan, thin] (2,0,0) -- (0,2/3,2/3);
\draw[cyan, thin] (0,2,0) -- (2/3,0,2/3);
\draw[cyan, thin] (0,0,2) -- (2/3,2/3,0);

\filldraw[red] (1/2,1/2,1/2) circle (1.5pt) node[anchor=south] {\scriptsize{$b$}};
\filldraw[cyan] (0,2/3,2/3) circle (0.5pt);
\filldraw[cyan] (2/3,0,2/3) circle (0.5pt);
\filldraw[cyan] (2/3,2/3,0) circle (0.5pt);

\end{tikzpicture}  

        \caption{Twisted Veronese models $\V_{2,3}$ and $\V_{3,2}$, respectvely.}
        \label{fig:twisted-veronese}
    \end{figure}
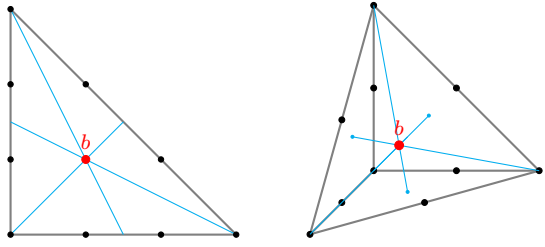

The maximum divergence from $\V_{2,3}$ is $2\log 3$, achieved at the unique point $v\in\Delta_9$, uniformly supported on $3e_1,3e_2,$ and $3e_3$, i.e. $v_{300}=v_{030}=v_{003}=1/3$ and all other coordinates of $v$ are 0. Note that $v$ is a vertex of the logarithmic Voronoi polytope  $Q_b$ corresponding to the centroid $b$ of $\conv(A)$, i.e. $Aq=b=(1,1,1)$. The point $q$ can be computed using (\ref{eq:mle-twisted-veronese}), so $q=(1/3)^3(1,3,3,3,6,3,1,3,3,1)$. The maximum divergence is 
\begin{align*}
D(v||q)&=v_{300}\log(v_{300}/q_{300})+v_{030}\log(v_{030}/q_{030})+v_{003}\log(v_{003}/q_{003})\\
&=1/3\log(3^3/3)+1/3\log(3^3/3)+1/3\log(3^3/3)=2\log3.
\end{align*}
\end{example}

\begin{example}[$d=3,m=2$]For the twisted Veronese model $\V_{3,2}$, the maximum divergence is $\log 4$. It is achieved at 10 different vertices of the logarithmic Voronoi polytope at the point $q=1/16 (1, 2, 1, 2,2, 1, 2,2,2, 1)$. Note that $Aq=b=(1/2,1/2,1/2,1/2)$, which is again the centroid of $\conv(A)$. One of such vertices is $v=(1/4, 0, 1/4, 0, 0, 0, 0, 0, 1/2, 0)$, so the divergence is $D(v||q)=1/4\log 4+1/4\log 4+1/2\log 4=\log 4$. The polytope $\conv(A)$ for this model is shown in Figure \ref{fig:twisted-veronese} on the right. Each of the 10 maximizers arises from one of the 10 permutations in $S_4$ of order at most two. This will follow from the proof of Theorem \ref{thm:twisted-veronese-max-divergence} in the next section.
\end{example}

The formula for the maximum divergence from a general model $\V_{d,m}$ as well as the full description of maximizers were given in  \cite{juricek}. We summarize these results in the theorem below. A more detailed discussion about the maximizers will be presented in the next section.

\begin{theorem}\cite[Theorem 1.1]{juricek}\label{thm:twisted-veronese-max-divergence}
   The maximum divergence to $\V_{d,m}$ equals $(m-1)\log (d+1)$. It is achieved at some vertices of the unique logarithmic Voronoi polytope $Q_b$ where $b=\left(\frac{m}{d+1},\ldots,\frac{m}{d+1}\right)$. There is a unique vertex 
   of this polytope maximizing divergence if and only if $m>2$.
\end{theorem}

\subsection{Box model}
In this section we consider a generalization of the twisted Veronese model. Suppose we have $k>1$ groups of random variables, with $a_i$ independent identically distributed $(d+1)$-ary random variables with state space $\{0,\ldots,d\}$ in the $i$th group for $i\in[k]$. Let $s_{i\ell}$ be the probability of state $\ell$ in the group $i$ and let $p_{(j_{10}\ldots j_{1d}),\ldots,(j_{k0}\ldots j_{kd})}$ denote the probability of observing exactly $j_{i\ell}$ occurrences of state $\ell$ in the group $i$. Hence, $p_{(j_{10}\ldots j_{1d}),\ldots,(j_{k0}\ldots j_{kd})}=\prod_{i=1}^k\binom{a_i}{j_{i0}\ldots j_{id}}s_{i0}^{j_{i_0}}\ldots s_{id}^{j_{id}}$ and we have a $kd$-dimensional toric model parametrized as
$$\Delta_{d}\times \ldots \times \Delta_{d}\to \Delta_{n-1}: ((s_{10},\ldots, s_{1d}),\ldots,(s_{k0},\ldots, s_{kd}))\mapsto (p_{(j_{10}\ldots j_{1d}),\ldots,(j_{k0}\ldots j_{kd})}: 0 \leq j_{i\ell}\leq a_i),$$
where $n=\prod_{i=1}^k\binom{a_i+d}{d}$. The columns of the corresponding matrix $A$ are naturally identified with the nonnegative integer solutions to the linear system $j_{i0}+\ldots+{j_{id}}=a_i$ for $i\in[k]$. We will refer to this model as the \textit{box model} motivated by the shape of 
$\conv(A)$ when $d=1$. We denote these models by $\B^{(d)}_{a_1,\ldots,a_k}$. The special case when $d=1, k=2$ was studied in \cite{Davies_2022}. The box model also has ML degree one, and hence the MLE can be written as a rational function of data. For $u\in\Delta_{n-1}$ such that $Au=b=((b_{10},\ldots, b_{1a_1}),\ldots, (b_{k0},\ldots, b_{ka_k}))$, the MLE is 
$$q_{(j_{10}\ldots j_{1d}),\ldots,(j_{k0}\ldots j_{kd})}=\prod_{i=1}^k \binom{a_i}{j_{i0},\ldots,j_{id}}\left(\frac{b_{i0}}{a_i}\right)^{j_{i0}}\ldots \left(\frac{b_{id}}{a_i}\right)^{j_{id}}. $$
\begin{theorem}\label{thm:box-max-divergence}
    The maximum divergence to the model $\B^{(d)}_{a_1,\ldots,a_k}$ equals $(a_1+\ldots+a_k-1)\log (d+1)$. It is achieved at $[(d+1)!]^{k-1}$ vertices of the unique logarithmic Voronoi polytope $Q_b$ such that $b=((\frac{a_1}{d+1},\ldots,\frac{a_1}{d+1}),\ldots,(\frac{a_k}{d+1},\ldots,\frac{a_k}{d+1}))$.
\end{theorem}

Our proof of Theorem \ref{thm:box-max-divergence} relies heavily on the methods used to prove Theorem \ref{thm:twisted-veronese-max-divergence} in \cite{juricek}. Before we present both proofs, we outline the general theory below.

Let $\F$ be a model inside the simplex $\Delta_{N-1}$. Let $S_N$ be the symmetric group of all permutations on $[N]$. This group acts on $\Delta_{N-1}$ by permuting the coordinates of the points in the simplex. Let $G$ be a subgroup of $S_N$. 
\begin{definition}
    The model $\F$ is said to be \textit{$G$-symmetrical} if for all $\sigma\in G$ and all $p\in \F$, we have $\sigma p\in \F$. A point $p\in\Delta_{N-1}$ is said to be \textit{$G$-exchangeable} if for all $\sigma\in G$, we have $\sigma p=p$. 
\end{definition}
Let $\F$ be a $G$-symmetrical model and let $\F/G$ denote the set of all orbits of $\F$ under the action of $G$. Let
$$\gamma_G:\F\to\F/G: p\mapsto \{\sigma p:\sigma\in G\}$$
be the map that sends an element in $\F$ to its orbit. Denote by $\E$ the closure of all $G$-exchangeable distributions in $\Delta_{N-1}$ and let $\M=\F\cap \E$ denote the induced model of all exchangeable distributions in $\F$. The following theorem holds in general.
\begin{theorem}\cite[Corollary 2.6]{juricek}\label{thm:juricek-general}
    Let $G$ be a subgroup of $S_N$ and let $\F\subseteq \Delta_{N-1}$ be a $G$-symmetrical family of distributions. If there exists a maximizer of $D_\F$ that is exchangeable, then $D(\gamma_G(\M))=D(\F)$ and all maximizers of $D_{\gamma_G(\M)}$ are of the form $\gamma_G(v)$ where $v$ is an exchangeable maximizer of $D_{\F}$.
    
    % The following are equivalent
    % \begin{enumerate}
    %     \item $v\in\E$ is a maximizer of $D_{\F}$ over $\E$;
    %     \item $v\in\E$ is a maximizer of $D_\M$ over $\E$;
    %     \item $\gamma_G(v)\in\Delta_{|\F/G|-1}$ is a maximizer of $D_{\gamma_G(\M)}$.
    % \end{enumerate}
\end{theorem}

\begin{proof}[Proof of Theorem \ref{thm:twisted-veronese-max-divergence} \cite{juricek}]
Let $\F$ be the independence model of $m$ $(d+1)$-ary random variables induced by $\Delta_d\times\cdots\times\Delta_d$. Then $\F\subseteq\Delta_{N-1}$ where $N=(d+1)^m$ and it is given 
by the following parametrization
\small
\begin{align}\label{param-independence}
    \varphi:((x_{10},\ldots, x_{1d}),\ldots,(x_{m0},\ldots, x_{md})) \mapsto (p_{j_1, \ldots, j_m} = x_{1j_1}x_{2,j_2}\cdots x_{mj_m}: \quad j_1,\ldots,j_m \in \{0, \ldots, d\}).
\end{align}
Let $G$ be the  subgroup of $S_N$ given by  
%which permutes the distributions among the $m$ components of $\F$, but fixes the coordinates of each individual distribution. That is, 
\small
\begin{align*}
  G=&\{\sigma_{\rho}\in S_N: \rho\in S_m\text{ and }\\ 
  & \sigma_\rho ((x_{10},\ldots, x_{1d}),\ldots,(x_{m0},\ldots, x_{md}))=((x_{\rho(1)0},\ldots, x_{\rho(1)d}),\ldots,(x_{\rho(m)0},\ldots, x_{\rho(m)d}))\}.
\end{align*}
Note that $G\cong S_m$ and it acts on each coordinate of $p\in\F$ as $\sigma p_{i_1\ldots i_m}=p_{i_{\sigma (1)}\ldots  i_{\sigma (m)}}$. Under this action, note that $\F$ is $G$-symmetrical and that the set of all $G$-exchangeable distributions in $\F$ is $\M=\{\varphi(x,x,\ldots,x): x\in\Delta_{d}\}$. Then the twisted Veronese model $\V_{d,m}$ can be identified with the set of all orbits of $\F$ coming from exchangeable distributions, i.e. $\V_{d,m}=\gamma_{G}(\M)$.

Since $\F$ is an independence model, we know that $D(\F)=(m-1)\log(d+1)$ and all of its maximizers are given in Theorem \ref{thm:multi-information-max}. Denote each such maximizer by $v_{\sigma_2,\ldots,\sigma_m}:=\frac{1}{d+1}\sum_{j=0}^{d}\delta_{j,\sigma_2(j),\ldots,\sigma_m(j)}$. Note that $v=v_{\id,\ldots,\id}$ is the distribution in $\F$ such that $v_{jj\ldots j}=\frac{1}{d+1}$ for each $j\in\{0,\ldots,d\}$ and 0 otherwise. By  Theorem \ref{thm:juricek-general}, it then follows that $D(\V_{d,m})=D(\gamma_G(\M))=D(\F)=(m-1)\log(d+1)$, as desired. Moreover, $w=\gamma_G(v)$ is a maximizer of $D_{\V_{d,m}}$. Explicitly, it is given as $w_x=\frac{1}{d+1}$ if $x=me_j$ for $j\in\{0,\ldots,d\}$ and $0$ otherwise.

If $m>2$, we claim that $w$ is the unique maximizer. Indeed, let $v=v_{\sigma_2,\ldots,\sigma_m}$ be another exchangeable maximizer of $\F$. Without loss of generality, assume $\sigma_2\neq \id$, so there is some $j$ such that $\sigma_2(j)\neq j$. Since $(j,\sigma_2(j),j_3,\ldots,j_{m})\in\supp(v)$ for some $j_3,\ldots j_{m}\in\{0,\ldots,d\}$, it has to be the case that $(\sigma_2(j),j,j_3,\ldots,j_{m})\in\supp(v)$, since $v$ is exchangeable and has to have the same value in both coordinates. But since $m>2$, it then follows that $j_k=\sigma_k(j)=\sigma_k(\sigma_2(j))$, so $\sigma_k$ is not injective for every $k\geq 3$, a contradiction. 

If $m=2$, then let $\sigma=\sigma_2$ and note that if $\sigma^2\neq \id$, then $v_{j k}\in \supp(v)$, but $v_{j k}=0$ for some $j,k\in\{0,\ldots, d\}$, which would contradict exchangeability of $v$. Hence, every maximizer of $\V_{d,2}$ is of the form $w=\frac{1}{d+1}\sum_{j=0}^d\delta_{e_j}+\delta_{e_{\sigma(j)}}$ for some $\sigma\in S_{d+1}$ of order at most two. Every nonzero coordinate of $w$ is thus either $\frac{1}{d+1}$ or $\frac{2}{d+1}$. The number of maximizers in this case is the number of permutations in $S_{d+1}$ of order at most two. Note that for each of the maximizers, we have $Av=\left(\frac{m}{d+1},\ldots,\frac{m}{d+1}\right)$ and hence they all lie in the same logarithmic Voronoi polytope at $q\in\V_{d,m}$, corresponding to the centroid of $\conv(A)$.
\end{proof}

\begin{proof}[Proof of Theorem \ref{thm:box-max-divergence}]
Let $\F$ be the independence model of $a_1+\cdots+a_k$ $(d+1)$-ary random variables divided into $k$ groups, induced by $$\underbrace{(\Delta_d\times\ldots\times \Delta_d)}_{a_1}\times\cdots\times\underbrace{(\Delta_d\times\ldots\times \Delta_d)}_{a_k}.$$ Then $\F\subseteq\Delta_{N-1}$ where $N=(d+1)^{a_1+\ldots+a_k}$ and has the parametrization $\varphi$ like (\ref{param-independence}), except each probability $p_{\bullet}$ factors as a product of ${a_1+\ldots+ a_k}$ parameters. Let $G$ be a subgroup of $S_N$ defined as
\small
\begin{align*}
G&=\{\sigma_{\rho_1,\ldots,\rho_k}\in S_N: \rho_i\in S_{a_i} \text{ and }\\
&\sigma_{\rho_1,\ldots,\rho_k}((y^{(1)}_1,\ldots,y^{(1)}_{a_1}),\ldots,(y^{(k)}_1,\ldots,y^{(k)}_{a_k}))=((y^{(1)}_{\rho_1(1)},\ldots,y^{(1)}_{\rho_1(a_1)}),\ldots,(y^{(k)}_{\rho_k(1)},\ldots,y^{(k)}_{\rho_k(a_k)}))\},
\end{align*}
where $y^{(i)}_{j}=(x^{(i)}_{j1},\ldots, x^{(i)}_{jd})\in\Delta_d$. Note that $G\cong S_{a_1}\times\ldots\times S_{a_k}$.
% $G$ acts on each coordinate of $\F$ as $\sigma_{\rho_1,\ldots,\rho_k}p_{\left(j_1^{(1)},\ldots,j_{a_1}^{(1)}
% \right),\ldots,\left(j_1^{(k)},\ldots,j_{a_k}^{(k)}\right)}=p_{\left(j_{\rho_1(1)}^{(1)},\ldots,j_{\rho_1(a_1)}^{(1)}\right),\ldots,\left(j_{\rho_k(1)}^{(k)},\ldots,j_{\rho_k(a_k)}^{(k)}\right)}$. \

Under this action, $\F$ is $G$-symmetrical and the set of all $G$-exchangeable distributions in $\F$ is $\M=\{\varphi((x^{(1)},\ldots,x^{(1)}),\ldots, (x^{(k)},\ldots,x^{(k)})): x^{(i)}\in\Delta_{d}\text{ for all }i\in[k]\}.$ The box model $\B^{(d)}_{a_1,\ldots,a_k}$ is then identified with the set of all orbits of $\F$ coming from exchangeable distributions, i.e. $\B^{(d)}_{a_1,\ldots,a_k}=\gamma_G(\M)$.

Since $\F$ is again an independence model, we know that $D(\F)=(a_1+\ldots+a_k-1)\log(d+1)$ from Theorem \ref{thm:multi-information-max}. Denote each maximizer by $v_{\sigma^{(1)}_2,\ldots, \sigma^{(1)}_{a_1},\ldots,\sigma^{(k)}_1,\ldots, \sigma^{(k)}_{a_1}}:=\frac{1}{d+1}\sum_{j=0}^{d}\delta_{j,\sigma^{(1)}_2(j),\ldots,\sigma^{(1)}_{a_1}(j),\ldots, \sigma^{(k)}_1(j),\ldots,\sigma^{(k)}_{a_k}(j)}$. First let $\sigma^{(1)}_{2}=\ldots=\sigma^{(1)}_{a_1}=\pi_1=\id$ and $\sigma^{(i)}_{1}=\ldots=\sigma^{(i)}_{a_i}=\pi_i$ for some $\pi_i\in S_{d+1}$ for all $i>1$. Then $v$ is a $G$-exchangeable maximizer of $\F$, and is explicitly given as $v_{(j\ldots j),(\pi_2(j)\ldots \pi_2(j)),\ldots,(\pi_k(j)\ldots \pi_k(j))}=\frac{1}{d+1}$ for any choice of $j\in \{0,\ldots,d\}$ and $0$ otherwise. The image of this maximizer under $\gamma$ is then $w=\frac{1}{d+1}\sum_{j=0}^d \delta_{\bigtimes_{i=1}^k{a_ie_{\pi_i(j)}}}$, where $\bigtimes$ denotes the Cartesian product of vectors in $\RR^{d+1}$. By Theorem \ref{thm:juricek-general}, $w$ is a maximizer of $\B^{(d)}_{a_1\ldots a_k}$. There are exactly $[(d+1)!]^{k-1}$ such maximizers: one for every choice of $(\pi_2,\ldots,\pi_k)\in S_{d+1}\times \cdots \times S_{d+1}$. Note also that for each such maximizer $w$, we have $Aw=((\frac{a_1}{d+1},\ldots,\frac{a_1}{d+1}),\ldots,(\frac{a_k}{d+1},\ldots,\frac{a_k}{d+1}))$, so all of them are the vertices of the logarithmic Voronoi polytope at the point $q\in\B^{(d)}_{a_1,\ldots, a_k}$ corresponding to the centroid of $\conv(A)$. 

We claim that there are no other maximizers of $\B^{(d)}_{a_1\ldots a_k}$. Indeed, if $a_i>2$ for all $i\in[k]$, then there are no other $G$-exchangeable maximizers of $\F$ by the proof of Theorem \ref{thm:twisted-veronese-max-divergence}. Indeed, if $a_1=a_2=1$ and $k=2$, then all maximizers are of the form discussed in the previous paragraph. If $a_i=2$ for some $i\in[k]$, without loss of generality assume that $a_1=2$ and that $v$ is a $G$-exchangeable maximizer of $\F$ with $\pi(j)=\sigma^{(1)}_2(j)\neq j$ for some $j$. If $v$ is $G$-exchangeable, it has to be the case that there are some values $j_3,\ldots,j_{a_1+\ldots+a_k}$ such that both $(j,\pi(j),j_3,\ldots,j_{a_1+\ldots+a_k})$ and $(\pi(j),j,j_3,\ldots,j_{a_1+\ldots+a_k})$ are in $\supp(v)$. But then $j_{3}=\sigma^{(2)}_1(j)=\sigma^{(2)}_1(\pi(j))$, a contradiction to the injectivity of $\pi$. Hence, there are no other maximizers.
\end{proof}

When $d=1$, the maximizers of the box model $\B^{(1)}_{a_1,\ldots,a_k}$ have the following nice geometric interpretation. 

\begin{cor}
    The maximum divergence from the box model $\B^{(1)}_{a_1,\ldots,a_k}$ equals $(a_1+\ldots+a_k-1)\log 2$. The maximum divergence is achieved at $2^{k-1}$ vertices of the unique logarithmic Voronoi polytope. These vertices correspond to the main diagonals of $\conv(A)$.
\end{cor}

% \begin{cor}
%     The maximum divergence from the box model $\B^{(2)}_{a_1,\ldots,a_k}$ equals $(a_1+\ldots+a_k-1)\log 3$. The maximum divergence is achieved at $4^r2^{k-1}$ vertices of the unique logarithmic Voronoi polytope, where $r=|\{i\in[k]: a_i=2\}|$.
% \end{cor}

\begin{example}[$d=1,k=3$]
Consider the box model $\B^{(1)}_{3,3,2}$. It is a $3$-dimensional model inside~$\Delta_{47}$. The columns of the corresponding matrix $A$ can be identified with the lattice points $\{(i,j,k)\in\ZZ^3: 0\leq i,j\leq 3, 0\leq k\leq 2 \}$. The maximum divergence of this model is $7\log 2$ and it is achieved at four vertices of the logarithmic Voronoi polytope $Q_b$  corresponding to the point $b=(3/2,3/2,1)$ in $\conv(A)$. This is illustrated in Figure \ref{fig:box-model-332}.

\begin{figure}[H]
    \centering
    \begin{tikzpicture}[scale=1.1]
\draw[gray, thick] (0,0,0) -- (0,3,0);
\draw[gray, thick] (0,0,0) -- (3,0,0);
\draw[gray, thick] (0,0,0) -- (0,0,2);
\draw[gray, thick] (3,0,2) -- (3,3,2);
\draw[gray, thick] (3,0,2) -- (0,0,2);
\draw[gray, thick] (3,0,2) -- (3,0,0);
\draw[gray, thick] (3,3,2) -- (0,3,2);
\draw[gray, thick] (3,3,2) -- (3,3,0);
\draw[gray, thick] (0,0,2) -- (0,3,2);
\draw[gray, thick] (0,3,0) -- (0,3,2);
\draw[gray, thick] (0,3,0) -- (3,3,0);
\draw[gray, thick] (3,0,0) -- (3,3,0);

\filldraw[black] (0,0,0) circle (1 pt) node[anchor=north] {\scriptsize{$(0,0,0)$}};
\filldraw[black] (3,0,0) circle (1 pt) node[anchor=west] {\scriptsize{$(3,0,0)$}};
\filldraw[black] (0,3,0) circle (1 pt) node[anchor=south] {\scriptsize{$(0,3,0)$}};
\filldraw[black] (0,0,2) circle (1 pt) node[anchor=north] {\scriptsize{$(0,0,2)$}};

\draw[cyan, thin] (0,0,0) -- (3,3,2);
\draw[cyan, thin] (3,0,0) -- (0,3,2);
\draw[cyan, thin] (0,3,0) -- (3,0,2);
\draw[cyan, thin] (0,0,2) -- (3,3,0);

\filldraw[red] (3/2,3/2,1) circle (1.5pt) node[anchor=south] {\scriptsize{$b$}};

\end{tikzpicture}

    \caption{Chamber complex of the box model $\B^{(1)}_{3,3,2}$.}
    \label{fig:box-model-332}
\end{figure}
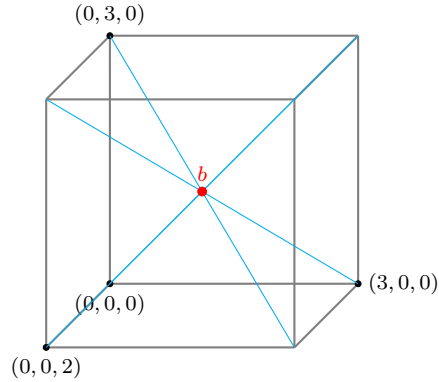

The four vertices giving $D(\B^{(1)}_{3,3,2})$ are supported on the main diagonals of $\conv(A)$. Explicitly, the four maximizers are 
\begin{align*}
&w_1=\frac{1}{2}\left(\delta_{(0,0,0)}+\delta_{(3,3,2)}\right),\;\;\; w_2=\frac{1}{2}\left(\delta_{(0,0,2)}+\delta_{(3,3,0)}\right),\;\;\;\\
&w_3=\frac{1}{2}\left(\delta_{(0,3,0)}+\delta_{(3,0,2)}\right),\;\;\;  w_4=\frac{1}{2}\left(\delta_{(3,0,0)}+\delta_{(0,3,2)}\right).\\
\end{align*}

\end{example}
\subsection{Trapezoid   model}
One interesting extension of the box model $\B_{a_1,a_2}^{(1)}$ is the trapezoid model, which we discuss in this section. It also has ML degree one \cite{Davies_2022}. Fix some positive integers $a,b,d$. Suppose we have two coins, with the probabilities of flipping heads being $s$ and $t$, respectively. First, we flip the second coin $b$ times, and record the number of heads  $j\in\{0,\ldots,b\}$. Then we flip the first coin $a$ times and record the number of heads, and then flip the first coin again $d(b-j)$ times and record the number of heads. The probability $p_{r,j}$ of getting exactly $r$ heads from the first coin and exactly $j$ heads from the second coin is then \begin{align}\label{eq:reparam-trapezoid}
p_{r,j}=c_{r,j}s^{r}(1-s)^{a+d(b-j)-r}t^j(1-t)^{b-j}
\end{align}
where $$c_{r,j}=\binom{b}{j}\sum_{\substack{0\leq i\leq a\\0\leq k\leq d(b-j)\\ i+k=r}} \binom{a}{i}\binom{d(b-j)}{k}.$$
Geometrically, this model is given by all lattice points inside the trapezoid with the vertices $\{(0,0),(0,b),(a,b),(a+db,0)\}$ with the weight of each point $(r,j)$ given by $c_{r,j}$. Hence we will call this model the \textit{trapezoid model} and denote it by $\T_{a,b,d}$. This is a 2-dimensional model inside $\Delta_{n-1}$, parametrized by
$(s,t)\mapsto (p_{r,j}: 0\leq j\leq b, 0 \leq r\leq a+d(b-j))$ where $n=\sum_{j=0}^b\sum_{r=0}^{a+d(b-j)}r$.

The MLE for any data point $u\in\Delta_n$ is a rational function of $u$. If $Au=(1,b_1,b_2)$, the MLE of $u$ is the point $q\in\T_{a,b,d}$ such that $Aq=(1,b_1,b_2)$. This point is given as 
$$q_{r,j}=c_{r,j}\left(\frac{b_1}{a+d(b-b_2)}\right)^r\left(1-\frac{b_1}{a+d(b-b_2)}\right)^{a+d(b-j)-r}\left(\frac{b_2}{b}\right)^j\left(1-\frac{b_2}{b}\right)^{b-j}.$$

\begin{example}[$a=b=d=1$]\label{ex:trapezoid111}
Consider the simplest nontrivial trapezoid model with $a=b=d=1$. It is a 2-dimensional toric model in $\Delta_4$ where
$$A=\begin{pmatrix}1&1&1&1&1\\0&0&1&1&2\\0&1&0&1&0\end{pmatrix}$$
and the chamber complex is shown in the middle of Figure \ref{fig:trapezoid-example}. Note that two-dimensional chambers will not contribute any projection points by Proposition \ref{prop:support-reasons-dimension}. There are only three interior vertices: $(1/2,1/2),(2/3,2/3)$ and $(1,1/2)$. All of them have triangles for logarithmic Voronoi polytopes, with supports $\{14,25,234\}, \{23,14,125\}$, and $\{34,25,145\}$, respectively. Therefore, all potential projection points will come from the edges. Running our algorithm on the ten interior edges, we find that there are exactly four projection vertices, corresponding to two different points on the model $\T_{1,1,1}$. The first point maps to $b_1=(2\sqrt{5}/5,\;-\sqrt{5}/5+1)\in\conv(A)$, and the two projection vertices are 
$$\left(\sqrt{5}/5, 0, 0, -\sqrt{5}/5 + 1, 0\right)\text{ and }
\left(0, -\sqrt{5}/5 + 1, 3\sqrt{5}/5- 1, 0, -2\sqrt{5}/5 + 1\right).$$
The latter vertex yields the maximum divergence $\log 2 + \log\left(\frac{1}{3-\sqrt{5}}\right)$. Similarly, the second point on the model maps to $b_2=(-\sqrt{5}/5+1,-\sqrt{5}/5+1)$, and the two projection vertices are 
$$\left(0, \sqrt{5}/5 + 1, 0, 0, \sqrt{5}/5\right) \text{ and }
\left(-2\sqrt{5}/5 + 1, 0, 3\sqrt{5}/5 - 1, -\sqrt{5}/5 + 1, 0\right).$$
The former vertex yields the same maximum divergence $\log 2 + \log\left(\frac{1}{3-\sqrt{5}}\right)$. 

The logarithmic Voronoi polytopes corresponding to both $b_1$ and $b_2$ are quadrilaterals. After projecting onto two-dimensional planes, we plot both polytopes in Figure \ref{fig:trapezoid-example}.

\end{example}
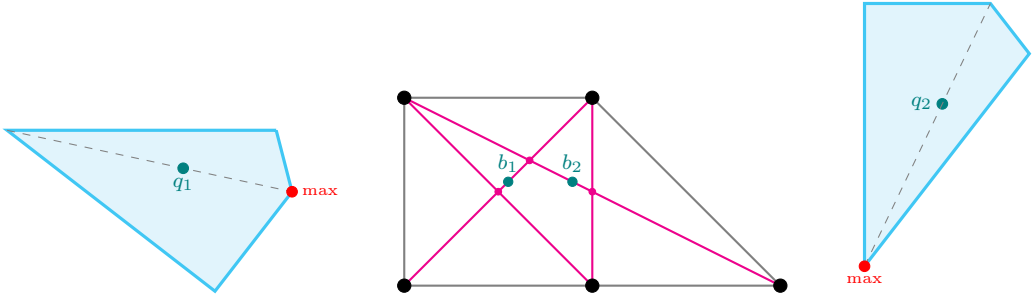
\begin{figure}[H]
    \centering
    \begin{tikzpicture}[scale=4]
\draw[color=cyan!60,fill=cyan!10, very thick] (0.39442719099991597, 0.1563582269576575) -- (-0.5, 0.15635822695765744) --(0.19098300562505255, -0.37887490770197735)  -- (0.4472135954999579, -0.048082638577173556) --(0.39442719099991597, 0.1563582269576575);
\draw[gray, dashed] (-0.5, 0.15635822695765744)--(0.4472135954999579, -0.048082638577173556);
\filldraw[teal] (0.08541019662496842, 0.030006823367684902) circle (0.5pt) node[anchor=north] {\scriptsize{$q_1$}};
\filldraw[red] (0.4472135954999579, -0.048082638577173556) circle (0.5pt)node[anchor=west] {\tiny max};
\end{tikzpicture}\;\;\;\;\;
\begin{tikzpicture}[scale=2.5]
\draw[gray, thick] (0,0) -- (0,1);
\draw[gray, thick] (0,0) -- (1,0);
\draw[gray, thick] (1,0) -- (2,0);
\draw[gray, thick] (0,1) -- (1,1);
\draw[gray, thick] (1,1) -- (2,0);
\draw[magenta, thick] (0,0) -- (1,1);
\draw[magenta, thick] (0,1) -- (1,0);
\draw[magenta, thick] (0,1) -- (2,0);
\draw[magenta, thick] (1,1) -- (1,0);
\filldraw[black] (0,0) circle (1pt);
\filldraw[black] (0,1) circle (1pt);
\filldraw[black] (1,0) circle (1pt);
\filldraw[black] (2,0) circle (1pt);
\filldraw[black] (1,1) circle (1pt);
\filldraw[magenta] (2/3,2/3) circle (0.5pt);
\filldraw[magenta] (1/2,1/2) circle (0.5pt);
\filldraw[magenta] (1,1/2) circle (0.5pt);
\filldraw[teal] (0.894427191, 0.5527864045) circle (0.7pt) node[anchor=south] {\scriptsize{$b_2$}};
\filldraw[teal] (0.5527864045, 0.5527864045) circle (0.7pt) node[anchor=south] {\scriptsize{$b_1$}};
\end{tikzpicture}  \;\;\;\;\;\vspace{5pt}\begin{tikzpicture}[scale=4]
\draw[color=cyan!60,fill=cyan!10, very thick] (-0.18257418578161505, -0.49103417581798225)-- (0.3651483716701109, 0.21607260540996018)--(0.2358486152762385, 0.3829978731112827)--(-0.18257418584767093, 0.38299787307966)--(-0.18257418578161505, -0.49103417581798225); 
\filldraw[teal] (0.07602532694007387, 0.049147337677014624) circle (0.5pt) node[anchor=east] {\scriptsize{$q_2$}};
\draw[gray, dashed] (-0.18257418578161505, -0.49103417581798225)--(0.2358486152762385, 0.3829978731112827);
\filldraw[red] (-0.18257418578161505, -0.49103417581798225) circle (0.5pt)node[anchor=north] {\tiny max};

\end{tikzpicture}
    \caption{The chamber complex and two logarithmic Voronoi polytopes that yield maximum~divergence.}
    \label{fig:trapezoid-example}
\end{figure}

\begin{theorem}\label{thm:trapezoid-upper-bound}
    The divergence to the trapezoid model $\T_{a,b,d}$ is bounded above by~$(a+bd+b)\log2$.
\end{theorem}
\begin{proof}
Fix a model $\T_{a,b,d}$ in $\Delta_{n-1}$, parametrized by $s$ and $t$. Let $u=(u_{r,j})\in\Delta_{n-1}$ be a general data vector. Then the log-likelihood function is
\begin{align*}
    \ell_u(p)=\sum u_{r,j} \log(p_{r,j})
    =\sum u_{r,j} \log c_{r,j}&+ \sum u_{r,j} [r \log s + (a+d(b-j)-r)\log (1-s)]\\
    &+\sum u_{r,j} [j\log(t)+(b-j)\log(1-t)].
\end{align*}
Taking the partial derivatives and solving for the parameters, we get that
$\hat{s}=\frac{\sum u_{r,j}r} {\sum u_{r,j}(a+d(b-j))}$ and $\hat{t}=\frac{\sum u_{r,j}j}{b}$. The MLE $q$ is obtained by plugging in these parameters into (\ref{eq:reparam-trapezoid}). Hence, the divergence function from the general point $u$ to the model $\T_{a,b,d}$ is
\begin{align*}
    D(u||\T_{a,b,d})&=\sum u_{r,j}\log(u_{r,j}/q_{r,j})
    =\underbrace{-H(u)-\sum u_{r,j }\log(c_{r,j})}_{\leq 0}\\
    &-\sum u_{r,j}r\log(\hat{s})-\sum u_{r,j}(a+(b-j)-r)\log(1-\hat{s})\\
    & -\sum u_{r,j}j\log(\hat{t})-\sum u_{r,j}(b-j)\log(1-\hat{t}),
\end{align*}
where $H(u)=-\sum_{i=1}^nu_i\log(u_i)$ is the entropy. Let $h(p)$ denote the entropy of a binary random variable with the probability of success $p$, i.e. $h(p)=-p\log(p)-(1-p)\log(1-p)$. Note that $h(p)$ always attains its maximum value at $p=1/2$. Therefore, we have
\begin{align*}
     D(u||\T_{a,b,d})&\leq \left(\sum u_{r,j}(a+d(b-j))\right)h(\hat{s})+ b\,h(\hat{t})\\
      &\leq  \left(\sum u_{r,j}(a+d(b-j)+b)\right)h(1/2)\\
    &=(a+db+b)\log 2 -d\sum u_{r,j}j\\
    &\leq(a+db+b)\log 2,
\end{align*}
as desired.
\end{proof}
Note that the polytope of the trapezoid model $\T_{a,b,d}$ sits in-between the polytopes of the two box models $\B^{(1)}_{a,b}$ and $\B^{(1)}_{a+db,b}$. However, the weights assigned to the lattice points are different. This presents the question of whether $D(\T_{a,b,d})$ is bounded by $D(\B^{(1)}_{a,b})$ and $D(\B^{(1)}_{a+db,b})$ below and above, respectively. Note that this is indeed the case for Example \ref{ex:trapezoid111}. We present the following conjecture.

\begin{conj}
    The divergence from the trapezoid model $\T_{a,b,d}$ is at least $(a+b-1)\log2$ and at most $(a+bd+b-1)\log2$. This upper bound is sharp if and only if $d=0$.
\end{conj}

In \cite{Davies_2022}, the authors present several families of 3-dimensional models that have ML degree one. We compute the maximum divergence for the simplest nontrivial examples in those families in the table below. 
\begin{center}
\begin{tabular}{|c|c|c|c|}
\hline
    \textbf{polytope} & \textbf{$\conv(A)$} & \textbf{$D(\M)$} & \textbf{notes} \\
\hline
\begin{minipage}{0.15\textwidth}\begin{tikzpicture}[every path/.style={>=latex},every node/.style={draw,circle,fill=black,scale=0.5},scale=0.6]
\begin{scope}[scale=3]
\draw (1.1,0.75) node[draw,circle,inner sep=2pt,fill](5){};
    \draw (0,0.75) node[draw,circle,inner sep=2pt,fill](4){};
    \draw (0,1.25) node[draw,circle,inner sep=2pt,fill](6){};
    \draw (0.25,1.35) node[draw,circle,inner sep=2pt,fill](7){};
     \draw (0.3,1.25) node[draw,circle,inner sep=2pt,fill](9){};
  \draw (0.33,1) node[draw,circle,inner sep=2pt,fill](1){};
  \draw (0.35,0.75) node[draw,circle,inner sep=2pt,fill](3){};
  \draw (0.8,1) node[draw,circle,inner sep=2pt,fill](2){};
  \draw[-,dotted] (1) edge[thick] (4);
  \draw[-] (4) edge[thick] (3);
  \draw[-,dotted] (2) edge[thick] (1);
  \draw[-] (3) edge[thick] (5);
  \draw[-] (2) edge[thick] (5);
  \draw[-,dotted] (1) edge[thick] (7);
  \draw[-] (2) edge[thick] (7);
  \draw[-] (4) edge[thick] (6);
  \draw[-] (5) edge[thick] (9);
   \draw[-] (6) edge[thick] (7);
  \draw[-] (7) edge[thick] (9);
  \draw[-] (6) edge[thick] (9);
  \end{scope}
 \end{tikzpicture} \end{minipage}& $\substack{\conv\{(0,0,0),(0,0,1),(0,1,0),\\(0,1,1),(1,1,0),(2,0,0)\}}$ & $2\log 2$ & conjectured\\
 \hline
\begin{minipage}{0.15\textwidth}\begin{tikzpicture}[every path/.style={>=latex},every node/.style={draw,circle,fill=black,scale=0.5}, scale=0.6]
\begin{scope}[scale=3]
\draw (1.2,0.75) node[draw,circle,inner sep=2pt,fill](5){};
    \draw (0,0.75) node[draw,circle,inner sep=2pt,fill](4){};
    \draw (0,1.25) node[draw,circle,inner sep=2pt,fill](6){};
  \draw (0.3,1) node[draw,circle,inner sep=2pt,fill](1){};
  \draw (0.6,0.75) node[draw,circle,inner sep=2pt,fill](3){};
  \draw (0.6,1) node[draw,circle,inner sep=2pt,fill](2){};
  \draw (0.5,1.25) node[draw,circle,inner sep=2pt,fill](7){};
  \draw[-,dotted] (1) edge[thick] (4);
  \draw[-] (4) edge[thick] (3);
  \draw[-,dotted] (2) edge[thick] (1);
  \draw[-] (3) edge[thick] (5);
  \draw[-,dotted] (2) edge[thick] (5);
  \draw[-,dotted] (1) edge[thick] (6);
  \draw[-,dotted] (2) edge[thick] (7);
  \draw[-] (6) edge[thick] (7);
  \draw[-] (4) edge[thick] (6);
  \draw[-] (5) edge[thick] (7);
  \end{scope}
 \end{tikzpicture}\end{minipage}& $\substack{\conv\{(0,0,0),(0,0,1),(0,1,0),\\(1,0,1),(1,1,0),(2,0,0)\}}$ & $\log 2 + \log\left(\frac{1}{3-\sqrt{5}}\right)$ &\\
 \hline
\begin{minipage}{0.15\textwidth}\begin{tikzpicture}[every path/.style={>=latex},every node/.style={draw,circle,fill=black,scale=0.5}, scale=0.6]
\begin{scope}[scale=3]
\draw (1.2,0.75) node[draw,circle,inner sep=2pt,fill](5){};
    \draw (0,0.75) node[draw,circle,inner sep=2pt,fill](4){};
    \draw (0,1.35) node[draw,circle,inner sep=2pt,fill](6){};
  \draw (0.15,1) node[draw,circle,inner sep=2pt,fill](1){};
  \draw (0.6,0.75) node[draw,circle,inner sep=2pt,fill](3){};
  \draw (0.55,1) node[draw,circle,inner sep=2pt,fill](2){};
  \draw[-,dotted] (1) edge[thick] (4);
  \draw[-] (4) edge[thick] (3);
  \draw[-,dotted] (2) edge[thick] (1);
  \draw[-] (3) edge[thick] (5);
  \draw[-,dotted] (2) edge[thick] (5);
  \draw[-,dotted] (1) edge[thick] (6);
  \draw[-,dotted] (2) edge[thick] (6);
  \draw[-] (3) edge[thick] (6);
  \draw[-] (4) edge[thick] (6);
  \draw[-] (5) edge[thick] (6);
  \end{scope}
 \end{tikzpicture}\end{minipage}& $\substack{\conv\{(0,0,0),(0,0,1),(0,1,0),\\ (1,1,0),(2,0,0)\}}$ & $\log 2 + \log\left(\frac{1}{3-\sqrt{5}}\right)$ & boundary\\
 \hline
 \begin{minipage}{0.15\textwidth}\begin{tikzpicture}[every path/.style={>=latex},every node/.style={draw,circle,fill=black,scale=0.5}, scale=0.6]
\begin{scope}[scale=3]
\draw (0.35,0.75) node[draw,circle,inner sep=2pt,fill](8){};
\draw (0.55,1) node[draw,circle,inner sep=2pt,fill](8){};
    \draw (0,0.75) node[draw,circle,inner sep=2pt,fill](4){};
    \draw (0,1.25) node[draw,circle,inner sep=2pt,fill](6){};
    \draw (0.2,1.35) node[draw,circle,inner sep=2pt,fill](7){};
  \draw (0.2,1) node[draw,circle,inner sep=2pt,fill](1){};
  \draw (0.7,0.75) node[draw,circle,inner sep=2pt,fill](3){};
  \draw (1,1) node[draw,circle,inner sep=2pt,fill](2){};
  \draw[-,dotted] (1) edge[thick] (4);
  \draw[-] (4) edge[thick] (3);
  \draw[-,dotted] (2) edge[thick] (1);
  \draw[-] (2) edge[thick] (3);
  \draw[-,dotted] (1) edge[thick] (7);
  \draw[-] (2) edge[thick] (7);
  \draw[-] (4) edge[thick] (6);
   \draw[-] (6) edge[thick] (7);
  \draw[-] (3) edge[thick] (6);
  \end{scope}
 \end{tikzpicture}\end{minipage}&$\substack{\conv\{(0,0,0),(0,0,1),(0,1,0),\\(0,1,1), (2,0,0),(2,1,0)\}}$ &$\log 3$ & conjectured \\
 \hline
 \begin{minipage}{0.15\textwidth}\begin{tikzpicture}[every path/.style={>=latex},every node/.style={draw,circle,fill=black,scale=0.5},scale=0.6]
\begin{scope}[scale=3]
    \draw (0,0.75) node[draw,circle,inner sep=2pt,fill](4){};
    \draw (0,1.25) node[draw,circle,inner sep=2pt,fill](6){};
  \draw (0.2,1) node[draw,circle,inner sep=2pt,fill](1){};
  \draw (0.8,0.75) node[draw,circle,inner sep=2pt,fill](3){};
  \draw (0.9,1) node[draw,circle,inner sep=2pt,fill](2){};
  \draw[-,dotted] (1) edge[thick] (4);
  \draw[-] (4) edge[thick] (3);
  \draw[-,dotted] (2) edge[thick] (1);
  \draw[-] (2) edge[thick] (3);
  \draw[-] (4) edge[thick] (6);
   \draw[-,dotted] (6) edge[thick] (1);
  \draw[-] (6) edge[thick] (2);
  \draw[-] (6) edge[thick] (3);
  \end{scope}
 \end{tikzpicture}\end{minipage}&$\substack{\conv\{(0,0,0),(0,0,1),(0,1,0),\\ (1,0,0),(1,1,0)\}}$ & $\log 2$& boundary\\
 \hline
 \begin{minipage}{0.15\textwidth}
\begin{tikzpicture}[every path/.style={>=latex},every node/.style={draw,circle,fill=black,scale=0.5},scale=0.6]
\begin{scope}[scale=3]
    \draw (0,0.75) node[draw,circle,inner sep=2pt,fill](4){};
    \draw (0,1.25) node[draw,circle,inner sep=2pt,fill](6){};
    \draw (0.3,1.5) node[draw,circle,inner sep=2pt,fill](7){};
     \draw (1,1.5) node[draw,circle,inner sep=2pt,fill](8){};
     \draw (0.7,1.25) node[draw,circle,inner sep=2pt,fill](9){};
  \draw (0.3,1) node[draw,circle,inner sep=2pt,fill](1){};
  \draw (0.7,0.75) node[draw,circle,inner sep=2pt,fill](3){};
  \draw (1,1) node[draw,circle,inner sep=2pt,fill](2){};
  \draw[-,dotted] (1) edge[thick] (4);
  \draw[-] (4) edge[thick] (3);
  \draw[-,dotted] (2) edge[thick] (1);
  \draw[-] (2) edge[thick] (3);
  \draw[-,dotted] (1) edge[thick] (7);
  \draw[-] (2) edge[thick] (8);
  \draw[-] (4) edge[thick] (6);
   \draw[-] (6) edge[thick] (7);
  \draw[-] (7) edge[thick] (8);
  \draw[-] (6) edge[thick] (9);
  \draw[-] (8) edge[thick] (9);
  \draw[-] (3) edge[thick] (9);
  \end{scope}
 \end{tikzpicture}\end{minipage}&$\substack{\conv\{(0,0,0),(0,0,1),(0,1,0),(1,0,0),\\ (0,0,1),(1,0,1),(1,1,0),(1,1,1)\}}$&$2\log 2$&box model\\
 \hline
 \begin{minipage}{0.15\textwidth}\begin{tikzpicture}[every path/.style={>=latex},every node/.style={draw,circle,fill=black,scale=0.5},scale=0.6]
\begin{scope}[scale=3]
    \draw (0,0.75) node[draw,circle,inner sep=2pt,fill](4){};
    \draw (0,1.25) node[draw,circle,inner sep=2pt,fill](6){};
    \draw (0.3,1.5) node[draw,circle,inner sep=2pt,fill](7){};
     \draw (0.7,1.25) node[draw,circle,inner sep=2pt,fill](9){};
  \draw (0.3,1) node[draw,circle,inner sep=2pt,fill](1){};
  \draw (0.7,0.75) node[draw,circle,inner sep=2pt,fill](3){};
  \draw[-,dotted] (1) edge[thick] (4);
  \draw[-] (4) edge[thick] (3);
  \draw[-,dotted] (1) edge[thick] (3);
  \draw[-,dotted] (1) edge[thick] (7);
  \draw[-] (4) edge[thick] (6);
   \draw[-] (6) edge[thick] (7);
  \draw[-] (7) edge[thick] (9);
  \draw[-] (6) edge[thick] (9);
  \draw[-] (3) edge[thick] (9);
  \end{scope}
 \end{tikzpicture}\end{minipage}&$\substack{\conv\{(0,0,0),(0,0,1),(0,1,0),\\(1,0,0),(0,0,1),(1,0,1)\}}$
 &$\log 2$&$2\times 3$ independence\\
  \hline
\end{tabular}
\end{center}
Interestingly, the second example in the table has infinitely many maximizers. For the third and the fifth models, all the maximizers of information divergence lie on the boundary of the simplex. For the conjectured examples, we were able to compute most of the ideals in step 3 of the algorithm, but not all. Some of the higher-dimensional ideals that arise in those cases are very complicated and we were not able to solve them using numerical tools. 

\section*{Acknowledgements} We wish to thank Guido Mont\'ufar for drawing our attention to the problem of maximizing information divergence and for encouraging us to study this problem from the perspective of logarithmic Voronoi polytopes. We are also grateful for the discussions we had and his valuable comments. Y.A. was partially supported by the National Science Foundation Graduate Research Fellowship under Grant No. DGE 2146752.

\bibliographystyle{plain}
\bibliography{references}
\end{document}